\documentclass[11pt]{article}

\usepackage[dvips,a4paper,hmargin=2.75cm,vmargin=3cm]{geometry}
\usepackage[utf8]{inputenc}
\usepackage{hyperref}
\usepackage{amssymb,amsmath,amsthm,mathtools}
\usepackage{enumitem}
\usepackage[mathcal]{euscript}
\usepackage[T1]{fontenc}
\usepackage{lmodern}
\usepackage{fontawesome}
\usepackage{bm}
\usepackage{microtype}
\usepackage{graphicx}
\usepackage[normalem]{ulem}
\usepackage{authblk}
\usepackage[yyyymmdd,24hr]{datetime}
\usepackage[small]{titlesec}
\usepackage{cite}

\hypersetup{colorlinks=true,linkcolor=blue,citecolor=blue,
			pdfinfo={
			Title   = {Asymptotics of the spectral data of perturbed Stark operators on the half-line
						with mixed boundary conditions},
			Author  = {Julio H. Toloza, Alfredo Uribe},
			Subject = {Manuscript}}
			}

\titleformat*{\subsection}{\bfseries\itshape}

\providecommand{\keywords}[1]
{
\noindent{\small
	\textbf{Keywords:} #1}
}

\providecommand{\msc}[1]
{
\noindent{\small
	\textbf{2010 MSC:} #1}
}

\newtheorem{theorem}{Theorem}[section]

\newtheorem{lemma}[theorem]{Lemma}
\newtheorem{corollary}[theorem]{Corollary}
\newtheorem{proposition}[theorem]{Proposition}
\newtheorem*{thm}{Theorem}
\newtheorem*{prop}{Proposition}

\theoremstyle{definition}

\newtheorem{examplex}[theorem]{Remark}
\newenvironment{remark}
{\pushQED{\qed}\examplex}
	{\popQED\endexamplex}

\newcommand{\R}{{\mathbb R}}
\newcommand{\N}{{\mathbb N}}

\newcommand{\C}{{\mathbb C}}
\newcommand{\K}{{\mathbb K}}

\newcommand{\cA}{\mathcal{A}}
\newcommand{\cB}{\mathcal{B}}
\newcommand{\cU}{\mathcal{U}}
\newcommand{\cD}{\mathcal{D}}

\newcommand{\cF}{\mathcal{F}}

\newcommand{\cJ}{\mathcal{J}}

\newcommand{\cV}{\mathcal{V}}

\newcommand{\bO}{{\mathcal{O}}}

\DeclareMathOperator{\ai}{Ai}
\DeclareMathOperator{\bi}{Bi}

\DeclareMathOperator{\ch}{ch}

\DeclareMathOperator*{\re}{Re}
\DeclareMathOperator*{\im}{Im}
\DeclareMathOperator{\dom}{\cD}

\newcommand{\abs}[1]{\left\lvert #1 \right\rvert}
\newcommand{\abss}[1]{\bigl\lvert #1 \bigr\rvert}
\newcommand{\norm}[1]{\left\lVert #1 \right\rVert}
\newcommand{\inner}[2]{\left\langle#1,#2\right\rangle}

\newcommand{\cc}[1]{\overline{#1}}

\newcommand{\mb}[1]{\heavysymbol{#1}}

\newcommand{\defeq}{\mathrel{\mathop:}=}

\newcommand{\sA}{\cA}
\newcommand{\bsA}{\mb{\mathfrak{A}}}
\newcommand{\om}{\omega}
\newcommand{\bom}{\underline{\mb{\omega}}}

\title {\bf Asymptotics of the spectral data of perturbed Stark operators on the half-line
		with mixed boundary conditions}

\author[1]{Julio H. Toloza%
		\thanks{\faEnvelopeO\, {julio.toloza@uns.edu.ar}}%
			}
\author[2]{Alfredo Uribe%
		\thanks{\faEnvelopeO\, {alfredo.uribe@uacm.edu.mx}}%
			}
\affil[1]{Instituto de Matemática (INMABB)\\
		 Departamento de Matemática\\
		 Universidad Nacional del Sur (UNS) - CONICET\\
		 Bahía Blanca\\
		 Argentina}
\affil[2]{Academia de Matemáticas\\
		 Universidad Autónoma de la Ciudad de México\\
		 Calz. Ermita Iztapalapa S/N\\
		 Col. Lomas de Zaragoza 09620, Ciudad de México\\
		 Mexico}

\date{}

\begin{document}

\maketitle

\begin{abstract}
We obtain sharp asymptotic formulas for the eigenvalues and norming constants of
Sturm-Liouville operators associated with the differential expression
\[
-\frac{d^2}{dx^2} + x + q(x),
\quad
x\in [0,\infty),
\]
together with the boundary condition $\varphi'(0) - b\varphi(0) =0$, $b\in\R$, where
\[
q\in \left\{ p\in L^2_\R(\R_+,(1+x)^r dx) : p'\in L^2_\R(\R_+,(1+x)^r dx)\right\}
\]
and $r>1$.
\end{abstract}

\bigskip

\msc{
34E10,	
34L15,	
81Q05,	
81Q10	
}


\section{Introduction and main results}
\label{sec:introduction}

Throughout this paper, we use the standard notation $\varphi' \defeq \partial_x \varphi$ and
$\dot{\varphi} \defeq \partial_z \varphi$.
By $\log$ we always mean the principal branch of the logarithm; consequently, $u^{q} \defeq
e^{q\log u}$ so, in particular, $u^{1/2} = \sqrt{u}$ whenever $u>0$. The inner product and norm in
$L^2(\R_+)$ are denoted $\inner{\cdot}{\cdot}_2$ and $\norm{\cdot}_2$, respectively.
Finally, in an order relation of the form $f(n) = \bO(g(n))$ we always assume $n\in\N$
and the limit $n\to\infty$.

Given $r>1$, let us define the Sobolev-type, real Hilbert space
\begin{equation*}
\bsA_r \defeq \left\{ q\in\sA_r\cap\text{AC}_\text{loc}([0,\infty)) : q'\in\sA_r\right\},
\quad \norm{q}_{\bsA_r}^2 \defeq \norm{q}_{\sA_r}^2 + \norm{q'}_{\sA_r}^2,
\end{equation*}
where
\begin{equation*}
\sA_r \defeq L^2_\R(\R_+,(1+x)^r dx),\quad
	\norm{q}_{\sA_r} \defeq \norm{q}_{L^2(\R_+,(1+x)^r dx)}.
\end{equation*}
This paper is concerned with the spectral analysis of self-adjoint operators associated
with the differential expression
\[
\tau_q \defeq -\frac{d^2}{dx^2} + x + q(x),
\quad
x\in [0,\infty),
\]
where $q\in\bsA_r$, $r>1$.
By standard theory (see e.g. \cite[Ch.~6]{weidmann}), $\tau_q$ is in the
limit-circle case at $0$ and in the limit-point case at $\infty$. Hence (the
closure of) the minimal operator $H_q'$ defined by $\tau_q$ is symmetric and has
deficiency indices $(1,1)$. The self-adjoint extensions of $H_q'$ are defined by imposing
the usual boundary condition at $x=0$. Namely, given $b\in\R\cup\{\infty\}$,
\begin{equation*}
\dom(H_{q,b})
	\defeq  \left\{\begin{gathered}
			\varphi\in L^2(\R_+) : \varphi,\varphi'\in\text{AC}_\text{loc}([0,\infty)),\
			\tau_q\varphi\in L^2(\R_+),
			\\
			\varphi'(0) - b\varphi(0) = 0\text{ if } b\in\R,\
			\varphi(0)=0 \text{ if } b=\infty
			\end{gathered}\right\},
\quad
H_{q,b}\varphi \defeq \tau_q\varphi.
\end{equation*}
Also, $H_{q,b}$ is semi-bounded from below. Moreover, it has
only simple, discrete spectrum, with a finite number of negative eigenvalues (if any).
As a side note, we observe that $\sA_r\subset L^1(R_+)$ whenever $r>1$ and
$\norm{q}_1\le (r-1)^{-1/2}\norm{q}_{\sA_r}$ \cite[Sec.~2]{toluri-2022}.

In what follows, we shall consider only the case $b\in\R$ since
the Dirichlet boundary condition ($b=\infty$) has been discussed elsewhere \cite{toluri-2022}.

Let $\psi(q,z,x)$ be the unique (up to a constant multiple) square-integrable solution to the
eigenvalue problem $\tau_q\varphi = z\varphi$, $z\in\C$. Let us define
\begin{equation*}
w(q,b,z) \defeq \psi'(q,z,0) - b\psi(q,z,0).
\end{equation*}
According to the Borg--Marchenko uniqueness theorem \cite{gesztesy},
$H_{q,b}$ can be uniquely determined from the spectral data consisting of the set of eigenvalues
\begin{equation*}
\left\{\lambda_n(q,b)\right\}_{n=1}^\infty
	= \left\{\lambda\in\R : w(q,b,\lambda) = 0\right\},
\end{equation*}
along with the set of (logarithmic) norming constants $\left\{\kappa_n(q,b)\right\}_{n=1}^\infty$
given by
\begin{equation}
\label{eq:norming-constants}
e^{\kappa_n(q,b)}
	\defeq
	  \frac{\abs{\psi(q,\lambda_n(q,b),0)}^2}{\norm{\psi(q,\lambda_n(q,b),\cdot)}^{2}_2}
	= \frac{\psi(q,\lambda_n(q,b),0)}{\dot{w}(q,b,\lambda_n(q,b))},
\end{equation}
where
the last expression in \eqref{eq:norming-constants} follows after applying the identity
\begin{equation}
\label{eq:the-identity}
\partial_x(\psi\,\dot{\psi}' - \psi'\dot{\psi}) = -\psi^2.
\end{equation}

Let us take a look at the case $q\equiv 0$.
The square-integrable solution to the equation $\tau_0\varphi=z\varphi$ is given by the Airy
function of the first kind $\ai$. Namely,
\begin{equation*}
\psi_0(z,x) = \sqrt{\pi}\ai(x-z),
\end{equation*}
where the inclusion of the constant $\sqrt{\pi}$ is merely for convenience.
The eigenvalues of $H_{b}\defeq H_{0,b}$ are therefore the solutions to the equation
$\ai'(-\lambda) - b \ai(-\lambda) = 0$. In Section~\ref{sec:prelim} we obtain asymptotic
formulas for the spectral
data of $H_b$ in terms of the zeros of $\ai'$ (the derivative of the function $\ai$) and $b\in\R$.
Recall that the standard notation for the zeros of $\ai'$ is $a_n'$ (here the prime is part of the notation
and does not denote a derivative), where where $a_{n+1}'<a_n'<0$ for every $n\in\N$ and
\begin{equation}
\label{eq:zeros-of-ai-prime}
- a_n' = \left(\tfrac32\pi(n-\tfrac34)\right)^{2/3} + \bO(n^{-4/3})
\end{equation}
(see e.g. \cite[\S 9.9(iv)]{nist}).
Thus, we have (see Prop.~\ref{prop:spectral-data-unperturbed-mixed-bc}):
\begin{prop}
Let $a_n'$ be the $n$-th zero of the function $\ai'$. Assume $b\in\R$. Then,
\begin{equation*}
\lambda_n(b)
		= - a_n' - \frac{b}{a_n'} + \bO(n^{-4/3})
\quad\text{and}\quad
\kappa_n(b)
	= - \log(-a_n') + \frac{b^2}{a_n'} + \bO(n^{-4/3}),
\end{equation*}
where the error terms are uniform on bounded subsets of $\R$.
\end{prop}

Next, we summarize the main results of this paper (Thm.~\ref{thm:eigenvalues} and
Thm.~\ref{thm:norming-constants}). They involve the auxiliary function
\begin{equation}
\label{eq:omega_r}
\omega_r(n) \defeq	\begin{cases}
					n^{-1/3}\log^{1/2}n & \text{if } r\in(1,2),
					\\
					n^{-1/3}			& \text{if } r\in[2,\infty).
					\end{cases}
\end{equation}
\begin{thm}
Assume that $q\in\bsA_r$ and $b\in\R$. Then,
\begin{equation*}
\lambda_n(q,b)
	= - a_n' + \pi \frac{\int_0^\infty \ai^2(x+a_n')q(x)dx}{(-a_n')^{1/2}}
		- \frac{b}{a_n'}
		+ \bO(n^{-1/3}\omega_r^2(n)),
\end{equation*}
where the error term is uniform for $(q,b)$ on bounded subsets of $\bsA_r\times\R$.
Also,
\begin{equation*}
\kappa_n(q,b)
	= - \log(-a_n')
		- 2\pi \frac{\int_0^\infty \ai(x+a_n')\ai'(x+a_n')q(x)dx}{(-a_n')^{1/2}}
		+ \frac{q(0) + b^2}{a_n'}
		+ \bO(n^{-1/6}\omega_r^2(n)),
\end{equation*}
where the error term is uniform for $(q,q(0),b)$ on bounded subsets of $\bsA_r\times\R\times\R$.
\end{thm}

The methods used in this work are based on those introduced
by Pöschel and Trubowitz in \cite{poeschel} in their treatment of the inverse Dirichlet problem
in a finite interval (see also \cite{dahlber,isaacson-2,isaacson-1}). Closely related are the results of
\cite{chelkak1,chelkak2,chelkak2.5,chelkak3}, where the inverse problem for the perturbed
harmonic oscillator is investigated (in the real line as well as in the half-line).
We refer to \cite{toluri-2022} for a detailed recount of
results concerning the spectral theory of one-dimensional Stark operators.

As to the organization of this paper: Section~\ref{sec:prelim} discusses the spectral data of the
unperturbed operator $H_{0,b}$, the behavior of certain sets of fundamental solutions to the
eigenvalue problem $(\tau_q-z)\varphi=0$, and their analytic properties in the sense of Fréchet.
The main statements about the eigenvalues are worked out in Section~\ref{sec:eigenvalues}. Finally,
the norming constants are treated in
Section~\ref{sec:norming-constants}. The Appendix contains some auxiliary
results mostly used in Section~\ref{sec:norming-constants}.


\section{Preliminaries}
\label{sec:prelim}

\subsection[The unperturbed problem]{The unperturbed problem}

The eigenvalue equation
\begin{equation*}
\tau_0\varphi = z\varphi \quad (z\in\C)
\end{equation*}
has two sets of linearly independent solutions that are relevant to this work.

One pair of solutions is
\begin{equation*}
\psi_0(z,x) \defeq \sqrt{\pi}\ai(x-z)
\quad\text{and}\quad
\theta_0(z,x) \defeq \sqrt{\pi}\bi(x-z),
\end{equation*}
where $\ai$ and $\bi$ denote the Airy functions of the first and of second kind, respectively;
we refer to \cite[Sec. 9]{nist} for a summary of their properties.
We note that $\psi_0(z,\cdot)\in L^2(\R_+)$ for every $z\in\C$. Moreover,
\[
W(\psi_0(z),\theta_0(z)) \defeq \psi_0(z,x)\theta_0'(z,x)-\psi_0'(z,x)\theta_0(z,x) \equiv 1.
\]
Also, we have the identities
\begin{equation}
\label{eq:more-bs}
\dot{\psi_0}(z,x) = - \psi_0'(z,x)
\quad\text{and}\quad
\dot{\theta_0}(z,x) = - \theta_0'(z,x).
\end{equation}
Furthermore, in terms of the auxiliary functions
\begin{equation*}
\sigma(w)\defeq 1 + \abs{w}^{1/4},
\quad
g_A(w) \defeq \exp(-\tfrac23\re w^{3/2})
\quad\text{and}\quad
g_B(w) \defeq 1/g_A(w),
\end{equation*}
one has the inequalities
\begin{gather}
\abs{\psi_0(z,x)}
	\le C_0\frac{g_A(x-z)}{\sigma{(x-z)}},
\quad
\abs{\psi_0'(z,x)}
	\le C_0 \sigma{(x-z)}g_A(x-z),\label{eq:vbe-psi0-prime}
\\[1mm]
\abs{\theta_0(z,x)}
	\le 2C_0\frac{g_B(x-z)}{\sigma{(x-z)}}\quad\text{and}\quad
\abs{\theta_0'(z,x)}
	\le 2C_0 \sigma{(x-z)}g_B(x-z),\nonumber
\end{gather}
where $C_0$ is a positive constant \cite[Lemma A.1]{toluri-2022}. Concerning the function $g_A$,
the following assertion holds true; see \cite[Lemma~2.2]{toluri-2022} for a proof.

\begin{lemma}
If $z\in\C\setminus\R$, then $g_A(x-z)$ is a decreasing function of $x\in\R_+$ and $g_A(x-z)\to 0$ as
$x\to\infty$. If $\lambda\in\R$, then $g_A(x-\lambda)=1$ if $x\in[0,\lambda]$ and monotonically
decreases to 0 if $x\in(\lambda,\infty)$.
\end{lemma}

Another pair of solutions is
\begin{equation*}
s_0(z,x) \defeq -\theta_0(z,0)\psi_0(z,x) + \psi_0(z,0)\theta_0(z,x),
\quad
c_0(z,x) \defeq \theta_0'(z,0)\psi_0(z,x) - \psi_0'(z,0)\theta_0(z,x).
\end{equation*}
Clearly,
\begin{equation*}
s_0(z,0) = c_0'(z,0) = 0,\quad s_0'(z,0) = c_0(z,0) = 1
\end{equation*}
so again $W(c_0(z),s_0(z))\equiv 1$. They obey the identities
\begin{gather}
\dot{s}_0(z,x) = c_0(z,x) - s_0'(z,x),
\quad
\dot{c}_0(z,x) = -z s_0(z,x) - c_0'(z,x),\label{eq:s-dot-c-dot}
\intertext{from which it follows that}
\dot{s}_0'(z,x) = c_0'(z,x) - (x-z)s_0(z,x),
\quad
\dot{c}_0'(z,x) = -z s_0'(z,x) - (x-z)c_0(z,x).\nonumber
\end{gather}
Also,
\begin{gather}
\label{eq:very-basic-estimates-2}
\abs{s_0(z,x)}
	\le 2C_0^2\frac{\ch(z,x)}{\sigma(z)\sigma(x-z)},
\quad
\abs{s_0'(z,x)}
	\le 2C_0^2\frac{\sigma(x-z)}{\sigma(z)}\ch(z,x),
\\
\abs{c_0(z,x)}
	\le 2C_0^2 \frac{\sigma(z)}{\sigma(x-z)}\ch(z,x),
\quad
\abs{c_0'(z,x)}
	\le 2C_0^2 \sigma(z)\sigma(x-z)\ch(z,x),\nonumber
\end{gather}
where
\begin{equation}
\label{eq:ch}
\ch(z,x)\defeq g_B(-z)g_A(x-z) + g_A(-z)g_B(x-z).
\end{equation}

The spectrum of $H_b\defeq H_{0,b}$ is the sequence of real numbers
$\{\lambda_n(b)\}_{n=1}^\infty$, arranged according to increasing values, whose elements solve
\begin{equation}
\label{eq:root-of-things}
\ai'(-\lambda) - b\ai(-\lambda) = 0,
\end{equation}
while their associated norming constants $\{\kappa_n(b)\}_{n=1}^\infty$, in view of
\eqref{eq:norming-constants}, are given by
\begin{equation*}
e^{\kappa_n(b)}
	= \frac{\ai(-\lambda_n(b))}
		{\lambda_n(b)\ai(-\lambda_n(b)) + b \ai'(-\lambda_n(b))}
	= \frac{1}{\lambda_n(b) + b^2}.
\end{equation*}
Notice that the last expression above is always strictly positive.

\begin{proposition}
\label{prop:spectral-data-unperturbed-mixed-bc}
Let $a_n'$ be the $n$-th zero of the function $\ai'$. Assume $b\in\R$. Then,
\begin{equation*}
\lambda_n(b)
		= - a_n' - \frac{b}{a_n'} + \bO(n^{-4/3})
\quad\text{and}\quad
\kappa_n(b)
	= - \log(-a_n') + \frac{b^2}{a_n'} + \bO(n^{-4/3}),
\end{equation*}
where the error terms are uniform on bounded subsets of $\R$.
\end{proposition}
\begin{proof}
Let $\{a_n\}_{n=1}^\infty$ be the set of zeros of the function $\ai$ (all of them being negative
as we can see in e.g. \cite[\S 9.9(iv)]{nist})
and define $a_0\defeq\infty$; observe that $\sigma(H_\infty)=\{-a_n\}_{n=1}^\infty$.
Also, let us define
\begin{equation*}
w(b,\lambda)\defeq \ai'(-\lambda) - b\ai(-\lambda).
\end{equation*}
Due to the interlacing property of the spectra of $H_b$ for
different values of $b$, we have
\[
\bigcup_{b\in\R\cup\{\infty\}}\{\lambda_n(b)\}_{n=1}^\infty=\R,
\quad
-a_{n-1} < \lambda_{n}(b) < -a_{n}
\quad\text{and}\quad
b\ne b' \implies \lambda_n(b)\ne \lambda_n(b').
\]
Therefore, for every $n\in\N$, we have a unique bijective function
$\lambda_{n}:\R\to(-a_{n-1},-a_{n})$ that satisfies $w(b,\lambda_{n}(b)) = 0$.
On the other hand,
\begin{equation*}
\left.\frac{\partial w}{\partial\lambda}\right|_{(b,\lambda_n(b))}
	= (\lambda_n(b) + b^2)\ai(-\lambda_n(b))\ne 0
\end{equation*}
so, by the Implicit Function Theorem, for every $b_0\in\R$ there exist a neighborhood $I_{b_0}$ and
a continuously differentiable function $\mu$ such that $\mu(b_0)=\lambda_{n}(b_0)$ and $w(b,\mu(b)) = 0$
for $b\in I_{b_0}$. Due to the uniqueness,
\begin{equation*}
\mu(b) = \lambda_n(b),\quad b\in I_{b_0},
\end{equation*}
which in turn implies that $\lambda_n$ is indeed continuously differentiable across its domain.
Moreover,
\[
\lambda_n'(b) = \frac{1}{\lambda_n(b) + b^2}.
\]
This fact allows us to write
\begin{equation}
\label{eq:merry-go}
\lambda_n(b) - \lambda_n(0)
	= \int_0^1\frac{d}{dt}\lambda_n(tb) dt
	= b\int_0^1 \frac{dt}{\lambda_n(tb) + (tb)^2}.
\end{equation}
Thus, assuming $n\ge 2$,
\begin{equation*}
\abs{\lambda_n(b) - \lambda_n(0)}
	\le \abs{b} \int_0^1 \frac{dt}{-a_{n-1} + (tb)^2}
	\le -\frac{\abs{b}}{a_{n-1}},
\end{equation*}
so far implying that
\[
\lambda_n(b) = -a_n' + \bO(n^{-2/3}),
\]
where the implicit constant in the error term is uniform whenever $b$ lies in bounded subsets of
$\R$. Resorting to \eqref{eq:merry-go} again, we obtain
\begin{equation*}
\lambda_n(b) - \lambda_n(0)
	= b\int_0^1 \frac{dt}{-a_n' + (tb)^2 + \bO(n^{-2/3})}
	= -\frac{b}{-a_n'}\int_0^1 \frac{dt}{1 + \bO(n^{-2/3})},
\end{equation*}
where the implicit constant in the error term is (again) uniform if $tb$ belongs to any bounded
subset of $\R$. The first assertion now follows immediately. Finally, the asymptotic formula for
$\lambda_n(b)$ implies
\begin{equation*}
\kappa_n(b) = - \log(-a_n') - \log\left(1 - \frac{b^2}{a_n'} + \bO(n^{-4/3})\right),
\end{equation*}
from which the second assertion follows.
\end{proof}

\subsection[Introducing a perturbation]{Introducing a perturbation}

From now on,
the complexification of the real Hilbert spaces $\sA_r$ and $\bsA_r$ are denoted
$\sA_r^\C$ and $\bsA_r^\C$, respectively.

\begin{remark}
\label{rem:omegas}
Given $(q,z)\in\sA_r^\C\times\C$, let us define
\begin{equation*}
\om(q,z) \defeq \int_0^\infty\frac{\abs{q(x)}}{\sqrt{1+\abs{x-z}}}\,dx.
\end{equation*}
As shown in \cite[Lemma~2.1]{toluri-2022}, this function obeys the inequality
\begin{equation*}
	\om(q,z) \le C_r\norm{q}_{\sA_r}\Omega_r(z),\quad 
	\text{where}\quad
\Omega_r(z)
	\defeq	\begin{dcases}
			\left(\frac{\log(2+\abs{z})}{2+\abs{z}}\right)^{1/2},& r\in(1,2),
			\\
			\left(2+\abs{z}\right)^{-1/2},& r\in[2,\infty),
			\end{dcases}
\end{equation*}
and $C_r$ is a positive constant.
For $q\in\bsA_r^\C$, we also define $\bom(q,z)\defeq \om(q,z) + \om(q',z)$.
\end{remark}

Provided that $q\in\sA_r^\C$, the eigenvalue equation
\begin{equation}
\label{eq:stark-equation}
\tau_q\varphi = z\varphi\quad (z\in\C)
\end{equation}
has linearly independent solutions $\psi(q,z,x)$ and
$\theta(q,z,x)$ such that $\psi(q,z,\cdot)\in L^2(\R_+)$,
$W(\psi(q,z),\theta(q,z))\equiv 1$,
\begin{equation}
\label{eq:about-Xi}
\psi(q,z,x) = \psi_0(z,x) + \varXi(q,z,x),
\quad
\abs{\varXi(q,z,x)}
	\le C\om(q,z)e^{C\om(q,z)}\frac{g_A(x-z)}{\sigma(x-z)},
\end{equation}
and
\begin{equation*}
\theta(q,z,x) = \theta_0(z,x) + \varGamma(q,z,x),
\quad
\abs{\varGamma(q,z,x)}
	\le C\om(q,z)e^{C\om(q,z)}\frac{g_B(x-z)}{\sigma(x-z)}.
\end{equation*}
Moreover,
\begin{equation}
\label{eq:about-Xi-prime}
\psi'(q,z,x) = \psi_0'(z,x) + \varXi'(q,z,x),
\quad
\abs{\varXi'(q,z,x)}
	\le C\om(q,z)e^{C\om(q,z)}\sigma(x-z) g_A(x-z),
\end{equation}
and
\begin{equation*}
	\theta'(q,z,x) = \theta_0'(z,x) + \varGamma'(q,z,x),
\quad
\abs{\varGamma'(q,z,x)}
	\le C\om(q,z)e^{C\om(q,z)}\sigma(x-z) g_B(x-z).
\end{equation*}
Also, $\psi(q,\cdot,x)$, $\psi'(q,\cdot,x)$, $\theta(q,\cdot,x)$ and $\theta'(q,\cdot,x)$
are entire functions for every $(q,x)\in\sA_r^\C\times\R_+$ (real entire if $q$ is
restricted to $\sA_r$). Furthermore, assuming $q\in\bsA_r^\C$, we have
\begin{equation}
\label{eq:about-Xi-dot}
\dot{\psi}(q,z,x) = - \psi_0'(z,x) + \dot{\varXi}(q,z,x),
\quad
\abs{\dot{\varXi}(q,z,x)}
	\le C\bom(q,z)e^{C\bom(q,z)} \sigma(x-z) g_A(x-z).
\end{equation}
Proofs of these assertions are found in \cite{toluri-2022}.

Another pair of linearly independent solutions is
\begin{align*}
s(q,z,x) &\defeq -\theta(q,z,0)\psi(q,z,x) + \psi(q,z,0)\theta(q,z,x),
	\\
c(q,z,x) &\defeq  \theta'(q,z,0)\psi(q,z,x) - \psi'(q,z,0)\theta(q,z,x).
\end{align*}
They clearly obey the boundary conditions
\begin{equation*}
	s(q,z,0) = c'(q,z,0) = 0,\quad s'(q,z,0) = c(q,z,0) = 1.
\end{equation*}
Characterizations of these solutions and their partial derivatives are provided next.

\begin{lemma}
\label{lem:about-c}
Assume $q\in\sA_r^\C$. Then:
\begin{enumerate}[label={(\roman*)}]
\item\label{it:c}
	The solution $c(q,z,x)$ admits the decomposition
	\begin{equation*}
	c(q,z,x) = c_0(z,x) + \varUpsilon_c(q,z,x),
	\end{equation*}
	where
	\begin{equation*}
	\abs{\varUpsilon_c(q,z,x)}
		\le C\om(q,z)e^{C\om(q,z)}
			\frac{\sigma(z)}{\sigma(x-z)}\ch(z,x).
	\end{equation*}

\item\label{it:c-prime}
	Also,
	\begin{equation*}
		c'(q,z,x) = c_0'(z,x) + \varUpsilon_c'(q,z,x),
	\end{equation*}
	where
	\begin{equation*}
		\abs{\varUpsilon_c'(q,z,x)}
			\le C\om(q,z)e^{C\om(q,z)}
				\sigma(z)\sigma(x-z)\ch(z,x).
	\end{equation*}
\end{enumerate}
Moreover, $c(q,\cdot,x)$ and $c'(q,\cdot,x)$ are real entire functions for every
$(q,x)\in\sA_r\times\R_+$.
\begin{enumerate}[resume*]
\item\label{it:c-dot}
	Furthermore, if $q\in\bsA_r^\C$, we have
	\begin{equation*}
		\dot{c}(q,z,x) = \left(q(0)-z\right) s_0(z,x) - c_0'(z,x) + \dot{\varUpsilon}_c(q,z,x),
	\end{equation*}
	where
\begin{equation*}
\abs{\dot{\varUpsilon}_c(q,z,x)}
	\le C\bom(q,z)e^{C\bom(q,z)}
	\left(\frac{\abs{q(0)}+\abs{z}}{\sigma(z)\sigma(x-z)}
		+ \sigma(z)\sigma(x-z)\right)\ch(z,x).
\end{equation*}
\end{enumerate}
\end{lemma}
\begin{proof}
\ref{it:c} Let us write $c(z,x)$ to denote $c(q,z,x)$. The starting point is the Volterra integral equation
\begin{equation}
\label{eq:c-int-eqn}
c(z,x) = c_0(z,x) + \int_0^x J_0(z,x,y)c(z,y)q(y) dy\quad (z\in\C),
\end{equation}
where  $J_0(z,x,y)\defeq \theta_0(z,x)\phi_0(z,y) - \phi_0(z,x)\theta_0(z,y)$.
Let $c_n(z,x)$ ($n\in\N$) be defined recursively by means of the equation
\begin{equation*}
	c_n(z,x) \defeq \int_0^x J_0(z,x,y)c_{n-1}(z,y) q(y)dy.
\end{equation*}
Given $x\in\R_+$, $c_0(\cdot,x)$ is an entire function, which in turn implies that
$c_n(\cdot,x)$ is an entire function for every $n\in\N$.
By applying \eqref{eq:vbe-psi0-prime} and \eqref{eq:very-basic-estimates-2} we obtain
\begin{equation*}
\label{eq:bound-c}
\abs{c_n(z,x)}
		\le \frac{3^n2^{n+1}}{n!}C_0^{2(n+1)}\frac{\sigma(z)}{\sigma(x-z)}\ch(z,x)
			\left(\int_0^x\frac{\abs{q(y)}}{\sigma(y-z)^2}dy\right)^n
				\quad (n\in\N\cup\{0\}).
\end{equation*}
However,
\begin{equation*}
\int_0^x\frac{\abs{q(y)}}{\sigma(y-z)^2}dy \le \om(q,z),
\end{equation*}
hence
\begin{equation*}
	c(z,x) = \sum_{n=0}^\infty c_n(z,x)
\end{equation*}
converges uniformly on bounded subsets of $\sA_r^\C\times\C\times\R_+$ and
$c(z,x)$ solves \eqref{eq:c-int-eqn}, therefore it is solution to the equation
$\tau_q\varphi = z\varphi$ with the stated boundary conditions.
Clearly, $c(q,\cdot,x)$ is an entire function that becomes real entire if $q\in\sA_r$.

The proof of \ref{it:c-prime} follows the same scheme. In this case
\begin{equation*}
	c'(z,x) = \sum_{n=0}^\infty c_n'(z,x), \quad
		\text{where} \quad c_n'(z,x) =
			\int_0^x \partial_x J_0(z,x,y)c_{n-1}(z,y) q(y)dy
\end{equation*}
and
\begin{equation*}
	\abs{c_n'(z,x)}
		\le \frac{3^n 2^{n+1}}{n!}C_0^{2(n+1)} \sigma(z) \sigma(x-z) \ch(z,x)
			\left(\int_0^x \frac{\abs{q(y)}}{\sigma(y-z)^2}dy\right)^n
				\quad (n\in\N\cup\{0\}).
\end{equation*}
Again, the convergence is uniform on bounded subsets of
$\sA_r^\C\times\C\times\R_+$.

Regarding \ref{it:c-dot}, note that $\dot{c}(z,x)$ satisfies
\begin{align*}
	\dot{c}(z,x) = \dot{c}_0(z,x)
		& - \int_0^x \partial_x J_0(z,x,y) c(z,y) q(y) dy
		\\
		& - \int_0^x \partial_y J_0(z,x,y) c(z,y) q(y) dy
				+ \int_0^x J_0(z,x,y) \dot{c}(z,y) q(y) dy.
\end{align*}
Integrating by parts, the last equation becomes
\begin{align*}
	\dot{c}(z,x)+c'(z,x) = \left( q(0)-z \right) s_0(z,x)
		& + \int_0^x J_0(z,x,y) c(z,y) q'(y) dy
		\\
		& + \int_0^x J_0(z,x,y) \left( \dot{c}(z,y)+c'(z,y) \right) q(y) dy.
\end{align*}
Define
\begin{equation*}
	\gamma_n(z,x) \defeq \int_0^x J_0(z,x,y) c_{n-1}(z,y) q'(y) dy
			 + \int_0^x J_0(z,x,y) \gamma_{n-1}(z,y) q(y) dy,
\end{equation*}
where
\begin{equation*}
	\gamma_0(z,x) \defeq \left( q(0)-z \right) s_0(z,x).
\end{equation*}
A recursive argument shows that
\begin{equation*}
\abs{\gamma_n(z,x)}
	\le \frac{3^n 2^{n+1}}{n!}C_0^{2(n+1)}
		\left(\frac{\abs{q(0)}+\abs{z}}{\sigma(z)} + n\sigma(z) \right)
		\frac{\ch(z,x)}{\sigma(x-z)}
		\left(\int_0^x \frac{\abs{q(y)}+\abs{q'(y)}}{\sigma(y-z)^2}dy\right)^n,
\end{equation*}
for all $n\in\N\cup\{0\}$. Note that
\begin{equation*}
\int_0^x \frac{\abs{q(y)}+\abs{q'(y)}}{\sigma(y-z)^2}dy
	\le \bom(q,z).
\end{equation*}
In this way,
\begin{equation*}
	\dot{c}(z,x) + c'(z,x) = \sum_{n=0}^\infty \gamma_n(z,x),
\end{equation*}
the convergence being uniform on bounded subsets of $\bsA_r^\C\times\C\times\R_+$.
Now define
\begin{equation*}
\dot{\varUpsilon}_c(q,z,x)
	\defeq \sum_{n=1}^\infty \gamma_n(z,x) - \varUpsilon_c'(q,z,x).
\end{equation*}
Then,
\begin{equation*}
\dot{c}(z,x)
	= \left(q(0) - z\right)s_0(z,x) - c_0'(z,x) + \dot{\varUpsilon}_c(q,z,x)
\end{equation*}
Clearly
\begin{equation*}
\sum_{n=1}^\infty \abs{\gamma_n(z,x)}
	\le C\bom(q,z)e^{C\bom(q,z)}
		\left(\frac{\abs{q(0)}+\abs{z}}{\sigma(z)} + \sigma(z)\right)\frac{\ch(z,x)}{\sigma(x-z)}.
\end{equation*}
The assertion follows after taking into account \ref{it:c-prime} and the fact that $\sigma(w)\ge 1$.
\end{proof}

\begin{lemma}
\label{lem:about-s}
Assume $q\in\sA_r^\C$. Then:
\begin{enumerate}[label={(\roman*)}]
\item\label{it:s}
	The solution $s(q,z,x)$ admits the decomposition
	\begin{equation*}
	s(q,z,x) = s_0(z,x) + \varUpsilon_s(q,z,x),
	\end{equation*}
	where
	\begin{equation*}
	\abs{\varUpsilon_s(q,z,x)}
		\le C\om(q,z)e^{C\om(q,z)}
			\frac{\ch(z,x)}{\sigma(z)\sigma(x-z)}.
	\end{equation*}

\item\label{it:s-prime}
	Also,
	\begin{equation*}
	s'(q,z,x) = s_0'(z,x) + \varUpsilon_s'(q,z,x),
	\end{equation*}
	where
	\begin{equation*}
	\abs{\varUpsilon_s'(q,z,x)}
		\le C\om(q,z)e^{C\om(q,z)}
			\frac{\sigma(x-z)}{\sigma(z)}\ch(z,x).
	\end{equation*}
\end{enumerate}
Moreover, $s(q,\cdot,x)$ and $s'(q,\cdot,x)$ are real entire functions for every
$(q,x)\in\sA_r\times\R_+$.
\begin{enumerate}[resume*]
\item\label{it:s-dot}
	Furthermore, if $q\in\bsA_r^\C$, we have
	\begin{equation*}
	\dot{s}(q,z,x) = c_0(z,x) - s_0'(z,x) + \dot{\varUpsilon}_s(q,z,x),
	\end{equation*}
	where
	\begin{equation*}
	\abs{\dot{\varUpsilon}_s(q,z,x)}
		\le C\bom(q,z)e^{C\bom(q,z)}
			\left(\frac{\sigma(x-z)}{\sigma(z)}+\frac{\sigma(z)}{\sigma(x-z)}\right)\ch(z,x).
	\end{equation*}
\end{enumerate}
\end{lemma}
\begin{proof}
It is analogous to the proof of the previous result, hence omitted.
\end{proof}

There is another solution of interest, namely,
\begin{equation}
\label{eq:the-other-function}
\phi(q,b,z,x) \defeq c(q,z,x) + b s(q,z,x),
\end{equation}
where $b\in\C$. Clearly, $\phi(q,b,z,0) = 1$ and $\phi'(q,b,z,0) = b$. It admits the
decomposition
\begin{equation*}
\phi(q,b,z,x) = \phi_0(b,z,x) + \varPhi(q,b,z,x),\quad
\varPhi(q,b,z,x) \defeq \varUpsilon_c(q,z,x) + b \varUpsilon_s(q,z,x),
\end{equation*}
where $\phi_0(b,z,x)\defeq c_0(z,x) + b s_0(z,x)$. Using to the estimates already discussed,
we have
\begin{equation*}
\abs{\phi_0(b,z,x)}
	\le 2 C_0^2(1+\abs{b})\left(\sigma(z) + \frac{1}{\sigma(z)}\right)\frac{\ch(z,x)}{\sigma(x-z)}
\end{equation*}
and
\begin{equation*}
\abs{\varPhi(q,b,z,x)}
	\le C(1+\abs{b})\om(q,z)e^{C\om(q,z)}
		\left(\sigma(z) + \frac{1}{\sigma(z)}\right)\frac{\ch(z,x)}{\sigma(x-z)}.
\end{equation*}
Also, for $\phi'_0$ and $\dot{\phi}_0 = -z s_0 - c'_0 + b c_0 - b s_0'$, we have
\begin{equation*}
\label{eq:estimate-prime-phi-0}
\abs{\phi'_0(b,z,x)}
	\le 2 C_0^2(1+\abs{b})\left(\sigma(z) + \frac{1}{\sigma(z)}\right)\sigma(x-z)\ch(z,x)
\end{equation*}
and
\begin{equation}
\label{eq:estimate-dot-phi-0}
\abs{\dot{\phi}_0(b,z,x)}
	\le 2 C_0^2 (1+\abs{b})
		 \left(\sigma(z)\sigma(x-z) + \frac{\abs{z}}{\sigma(z)\sigma(x-z)}\right) \ch(\lambda,x).
\end{equation}
Later, in Remark~\ref{rem:the-end}, we derive more refined expressions for $\phi$ and its partial
derivatives under de assumption $q\in\bsA_r^\C$.

\subsection{Analyticity}

Let $\cU$ be an open subset of a Hilbert space $\cB$ over a field $\K$. A map $f:\cU\to\K$
is Fréchet differentiable at $q\in\cU$ if there exists a linear functional $d_qf:\cB\to\K$
such that
\[
\lim_{v\to 0}\frac{\abs{f(q+v) - f(q) - d_qf(v)}}{\norm{v}_\cB} = 0.
\]
The map $f$ is continuously differentiable on $\cU$ if it is differentiable at every point
in $\cU$ and the resulting map $df:\cU\to L(\cB,\K)$ is continuous. If $\cB$ is a
Hilbert space over $\C$, then $f$ is analytic on an open subset $\cU$ of $\cB$ if it is
continuously differentiable there.
Now, let $\cB^\C$ be the complexification of a real Hilbert space $\cB$ and assume
$f:\cV\to\C$ differentiable at $q\in\cV$ (an open subset of $\cB^\C$). Then,
the gradient of $f$ at $q$ is the (unique) element
$\partial f/\partial q\in\cB^\C$ such that
\[
d_qf(v) = \inner{\cc{\frac{\partial f}{\partial q}}}{v}_{\cB}
\]
for all $v\in\cB^\C$; here $\inner{\cdot}{\cdot}_\cB$ denotes the inner product in $\cB^\C$.
Finally, consider a real Hilbert space $\cB$ and let
$\cU\subset\cB$ be open. We say that $f:\cU\to\R$ is real analytic on $\cU$ if for every $q\in\cU$
there exists $\cV_q\subset\cB^\C$ open and an analytic map $h_q:\cV_q\to\C$ such that
$f(v) = h_q(v)$ for all $v\in\cU\cap\cV_q$.

\begin{lemma}
$\psi(\cdot,z,x)$, $\psi'(\cdot,z,x)$, $\dot{\psi}(\cdot,z,x)$ and
$\dot{\psi}'(\cdot,z,x)$ are analytic maps from $\sA_r^\C$ to $\C$. Their
corresponding gradients are given by
\begin{align*}
\frac{\partial\psi}{\partial q(y)}(q,z,x)
	&= - \cc{J(q,z,x,y)\psi(q,z,y)}\chi_{[x,\infty)}(y)(1+y)^{-r},
	\\
\frac{\partial\psi'}{\partial q(y)}(q,z,x)
	&= - \cc{\partial_x J(q,z,x,y)\psi(q,z,y)}\chi_{[x,\infty)}(y)(1+y)^{-r},
	\\
\frac{\partial\dot{\psi}}{\partial q(y)}(q,z,x)
	&= - \cc{\left(\partial_z J(q,z,x,y)\psi(q,z,y)
		+ J(q,z,x,y)\dot{\psi}(q,z,y)\right)}\chi_{[x,\infty)}(y)(1+y)^{-r}
\intertext{and}
\frac{\partial\dot{\psi}'}{\partial q(y)}(q,z,x)
	&= - \cc{\left(\partial_z\partial_x J(q,z,x,y)\psi(q,z,y)
		+ \partial_x J(q,z,x,y)\dot{\psi}(q,z,y)\right)}\chi_{[x,\infty)}(y)(1+y)^{-r},
\end{align*}
where $\chi_\cJ$ stands for the characteristic function associated with a set $\cJ$.
\end{lemma}
\begin{proof}
A proof concerning $\psi(\cdot,z,x)$ and $\psi'(\cdot,z,x)$ is given in Lemma 4.1 of
\cite{toluri-2022}. The analyticity of $\dot{\psi}(\cdot,z,x)$ is shown Lemma 4.2 of \cite{toluri-2022},
whose proof can easily be modified to accommodate the assertion about $\dot{\psi}'(\cdot,z,x)$.
\end{proof}

\begin{remark}
\label{rem:everything-has-gradient}
Later in the next section we assume $(q,\lambda)\in\sA_r\times\R$, in which case
\begin{align*}
\frac{\partial\psi}{\partial q(y)}(q,\lambda,0)
	&=  s(q,\lambda,y)\psi(q,\lambda,y)(1+y)^{-r},
\\
\frac{\partial\psi'}{\partial q(y)}(q,\lambda,0)
	&= - c(q,\lambda,y)\psi(q,\lambda,y)(1+y)^{-r},
\\
\frac{\partial\dot{\psi}}{\partial q(y)}(q,\lambda,0)
	&= \left[\dot{s}(q,\lambda,y)\psi(q,\lambda,y)
		+ s(q,\lambda,y)\dot{\psi}(q,\lambda,y)\right](1+y)^{-r}
\intertext{and}
\frac{\partial\dot{\psi}'}{\partial q(y)}(q,\lambda,0)
	&= - \left[\dot{c}(q,\lambda,y)\psi(q,\lambda,y)
		+ c(q,\lambda,y)\dot{\psi}(q,\lambda,y)\right](1+y)^{-r}.\qedhere
\end{align*}
\end{remark}

Let us extend our definition of
\begin{equation}
\label{eq:w}
w(q,b,z) = \psi'(q,z,0) - b\psi(q,z,0)
\end{equation}
by allowing $b\in\C$.

\begin{corollary}
$w(\cdot,\cdot,z)$ and $\dot{w}(\cdot,\cdot,z)$ are analytic maps from $\sA_r^\C\times\C$ to $\C$.
Moreover,
\begin{align*}
\frac{\partial w}{\partial q(y)}(q,b,\lambda)
	&= - \phi(q,b,\lambda,y)\psi(q,\lambda,y)(1+y)^{-r}
\intertext{and}
\frac{\partial\dot{w}}{\partial q(y)}(q,b,\lambda)
	&= - \left[\dot{\phi}(q,b,\lambda,y)\psi(q,\lambda,y) +
				\phi(q,b,\lambda,y)\dot{\psi}(q,\lambda,y)\right](1+y)^{-r},
\end{align*}
whenever $(q,b,\lambda)\in\sA_r\times\R\times\R$; here $\phi$ is the function defined in
\eqref{eq:the-other-function}. Also,
\begin{equation*}
\frac{\partial w}{\partial b}(q,b,\lambda)
	= -\psi(q,\lambda,0)
\quad\text{and}\quad
\frac{\partial \dot{w}}{\partial b}(q,b,\lambda)
	= -\dot{\psi}(q,\lambda,0).
\end{equation*}
\end{corollary}


\section{The eigenvalues}
\label{sec:eigenvalues}

For $m,n\in\N$ and $m\ge 2$, let us define the contours
\begin{equation}
\label{eq:loops}
\cF^m = \left\{z\in\C: \abs{\xi} = \bigl(m-\tfrac54\bigr)\pi\right\},
\quad
\cF_n = \left\{z\in\C: \abs{\xi - \bigl(n-\tfrac34\bigr)\pi}=\tfrac{\pi}{2}\right\},
\end{equation}
where $\xi\defeq\frac23 z^{3/2}$.
It is easy to verify that $\cF_n$ encloses exactly one zero of $\ai'(-z)$, namely $-a_n'$,
for sufficiently large values of $n$.

\begin{lemma}
\label{lem:bound-for-gA-Neumann}
There exist $m_0, n_0\in\N$ such that, for every $m\ge m_0$ and $n\ge n_0$, the
following statement holds true:
\begin{equation*}
\sigma(z) g_A(-z) < 16\sqrt{\pi} \abs{\text{\rm Ai}'(-z)},
\end{equation*}
whenever $z\in\mathcal{F}^m$ or $z\in\mathcal{F}_n$.
\end{lemma}
\begin{proof}
It is analogous to the proof of Lemma~A.2 of \cite{toluri-2022}, hence omitted.
\end{proof}

In what follows, we use the abbreviation
\begin{equation*}
w(z) \defeq w(q,b,z),\quad
\psi_0(z) \defeq \psi_0(z,0),\quad
\psi(z) \defeq \psi(q,z,0),\quad
\om(z) \defeq \om(q,z).
\end{equation*}

\begin{lemma}
\label{lem:crude-asymp-eigenvalues}
Assume $(q,b)\in\sA_r\times\R$. Then, given $\epsilon>0$ arbitrarily small,
the eigenvalues of $H_{q,b}$ satisfy
\begin{equation*}
\lambda_n(q,b) = - a_n' + \bO(n^{-2/3+\epsilon}),
\end{equation*}
uniformly on bounded subsets of $\sA_r\times\R$.
\end{lemma}
\begin{proof}
Assume $\lambda\in\R$. Recalling \eqref{eq:vbe-psi0-prime} and \eqref{eq:about-Xi},
and taking into account that $\omega(\lambda)\to 0$ as $\lambda\to\infty$, we obtain
\begin{equation*}
\abs{w(\lambda) - \psi_0'(\lambda)}
	\le \abs{\psi'(\lambda) - \psi_0'(\lambda)} + \abs{b}\abs{\psi_0(\lambda)}
	\le C_1 \left(\om(\lambda)\sigma(\lambda) + \frac{\abs{b}}{\sigma(\lambda)}\right).
\end{equation*}
for some constant $C_1>0$. Given $\epsilon\in(0,1/6)$, set $\delta_n = 8(\frac32\pi n)^{-2/3+\epsilon}$.
Let $\cU\times \cJ$ be a bounded subset of $\sA_r\times\R$.
Clearly,
\begin{align*}
\om(\lambda)\sigma(\lambda)
	&\le C\norm{q}_{\sA_r}
		\left(\frac{\log(2+\abs{\lambda})}{(2+\abs{\lambda})^{3\epsilon}}\right)^{1/2}
		\frac{1 + \abs{\lambda}^{1/4}}{(1 + \abs{\lambda})^{1/2 -3\epsilon/2}}.
\end{align*}
Thus, in view of \eqref{eq:zeros-of-ai-prime}, there exists $n_1\in\N$ such that
\begin{equation*}
C_1\om(-a_n' \pm \delta_n)\sigma(-a_n' \pm \delta_n)
	\le \abs{-a_n' \pm \delta_n}^{-1/4+3\epsilon/2}
	\le 2(\tfrac32\pi n)^{-1/6+\epsilon},
\end{equation*}
for all $n\ge n_1$ and $(q,b)\in\cU\times\cJ$. Also, there exists $n_2\ge n_1$ such that
\begin{equation*}
\frac{C_1\abs{b}}{\sigma(-a_n' \pm \delta_n)}
	\le 2(\tfrac32\pi n)^{-1/6+\epsilon},
\end{equation*}
for all $n\ge n_2$ and $b\in\cJ$. Therefore,
\begin{equation*}
\abs{w(-a_n' \pm \delta_n) - \psi_0'(-a_n' \pm \delta_n)}
	\le 4(\tfrac32\pi n)^{-1/6+\epsilon}
\end{equation*}
for $n\ge n_2$ and $(q,b)\in\cU\times \cJ$. The proof will be mostly
complete once we prove that
\begin{equation*}
\abss{\psi_0'(-a_n' \pm \delta_n)} > 4 (\tfrac32\pi n)^{-1/6+\epsilon},
\end{equation*}
since the inequality
\begin{equation}
\label{eq:target}
\abs{w(-a_n' \pm \delta_n) - \psi_0'(-a_n' \pm \delta_n)}
	< \abs{\psi_0'(-a_n' \pm \delta_n)}
\end{equation}
will imply that $w(z)$ has a zero on each interval $(-a_n' - \delta_n,-a_n' + \delta_n)$
for sufficiently large $n$.

Let us note that, due to \eqref{eq:more-bs}, $\dot{\psi}_0'(\lambda) = \lambda \psi_0(\lambda)$.
Also, we recall the well-known formula (see e.g. \cite[\S 9.7]{nist})
\begin{equation*}
\ai(-\lambda)
	= \frac{1}{\sqrt{\pi}\lambda^{1/4}}
		\left[\cos\left(\tfrac23\lambda^{3/2}-\tfrac14\pi\right) + \bO(\lambda^{-3/2})\right],
	\quad \lambda\to\infty.
\end{equation*}
Now, consider a sequence $\{c_n\}_{n=1}^\infty\subset\R$ such that $\abs{c_n}\le\delta_n$. Then,
\begin{align*}
(-a_n' + c_n)\psi_0(-a_n' + c_n)
	&= (-a_n' + c_n)^{3/4}\left[\cos\left(\tfrac23(-a_n' + c_n)^{3/2}-\tfrac14\pi\right)
			+ \bO(n^{-1})\right]
	\\
	&=(-1)^{n+1}(\tfrac32\pi n)^{1/2}
			\left[1 + \bO(n^{-2/3+2\epsilon})\right].
\end{align*}
Therefore, there exists $n_3\ge n_2$ such that
\begin{equation*}
\abss{\dot{\psi}_0'(-a_n' + c_n)} > \tfrac12 (\tfrac32\pi n)^{1/2},
\end{equation*}
as long as $n\ge n_3$. As a consequence, due to the Mean Value Theorem,
\begin{equation*}
\abs{\psi_0'(-a_n' \pm \delta_n)}
	= \abss{\dot{\psi}_0'(-a_n' + c_n^\pm)}\abs{\delta_n}
	> 4 (\tfrac32\pi n)^{-1/6+\epsilon},
\end{equation*}
where $c^{+}_n\in (0,\delta_n)$ and $c^{-}_n\in (-\delta_n,0)$. Thus,
\eqref{eq:target} holds for all  $n\ge n_3$ and $(q,b)\in\cU\times \cJ$. As already mentioned,
this implies $w(z)$ has a zero on each interval $(-a_n' - \delta_n,-a_n' + \delta_n)$.

To eliminate the possibility of having more than one zero near every $-a_n'$, we apply Rouché's Theorem
combined with Lemma~\ref{lem:bound-for-gA-Neumann} on the contours $\mathcal{F}^N$ and $\mathcal{F}_n$
introduced in \eqref{eq:loops}, for $N$ sufficiently large and every $n>N$.
The specifics are rather straightforward, hence omitted (c.f. \cite[Lemma~5.1]{toluri-2022}).
\end{proof}

\begin{remark}
\label{rem:emergence-omega_r}
Lemma~\ref{lem:crude-asymp-eigenvalues} implies
\begin{equation*}
\om(q,\lambda_{n}(q,b)) \le C\omega_r(n)
\quad\text{and}\quad
\bom(q,\lambda_{n}(q,b)) \le C\omega_r(n)
\end{equation*}
uniformly on bounded subsets of $\sA_r\times\R$ and $\bsA_r\times\R$, respectively, where
$\omega_r(n)$ has been defined in \eqref{eq:omega_r}.
\end{remark}

\begin{proposition}
\label{lem:eingenvalue-is-real-analytic}
Given $n\in\N$ and $b\in\R$, $\lambda_n(\cdot,b):\sA_r\to\R$ is a real analytic map. Moreover,
\begin{equation*}
\frac{\partial\lambda_n}{\partial q(x)} = \eta^2_n(q,b,x) (1+x)^{-r},\quad
\text{where}\quad
\eta_n(q,b,x)
	\defeq \frac{\psi(q,\lambda_n(q,b),x)}{\norm{\psi(q,\lambda_n(q,b),\cdot)}_2}.
\end{equation*}
\end{proposition}
\begin{proof}
Given $b\in\R$, define $w:\sA_r^\C\times\C\to\C$ by the rule
\begin{equation*}
w(q,\lambda) \defeq \psi'(q,\lambda,0) - b \psi(q,\lambda,0).
\end{equation*}
Recalling \eqref{eq:the-identity}, we note that
\begin{equation}
\label{eq:this-again}
\norm{\psi(q,\lambda,\cdot)}^2_2
	= \int_0^\infty \psi^2(q,\lambda,x) dx
	= \psi(q,\lambda,0)\dot{\psi}'(q,\lambda,0) - \psi'(q,\lambda,0)\dot{\psi}(q,\lambda,0).
\end{equation}
Consider $q\in\sA_r^\C$ real-valued and $\lambda_q=\lambda_{n}(q,b)$. Then $w(q,\lambda_q) = 0$,
which in turn implies
\begin{equation*}
\psi(q,\lambda_q,0) \partial_\lambda w(q,\lambda_q)
	= \norm{\psi(q,\lambda_n(q,b),\cdot)}^2_2 \ne 0
\end{equation*}
due to \eqref{eq:this-again}.
Since also $\psi(q,\lambda_q,0)\ne 0$, we conclude that $\partial_\lambda w(q,\lambda_q)$ is (rather, can be
identified with) a linear isomorphism from $\C$ to $\C$.
Therefore, by the Implicit Function Theorem (see e.g. \cite[XIV \S2, Thm. 2.1]{lang} or
\cite[Appendix B]{poeschel}), there exists an open neighborhood $\cV\subset\sA_r^\C$ of $q$ and a
unique analytic map $\lambda:\cV\to\C$ such that
\begin{equation*}
w(v,\lambda(v)) = 0 \text{ for all } v\in\cV \text{ and } \lambda(q) = \lambda_{n}(q,b).
\end{equation*}
Consequently, $\lambda_n(\cdot,b):\sA_r\to\R$ is real analytic. Since
\begin{equation*}
\frac{\partial w(q,b,\lambda)}{\partial q(x)}
	= - \phi(q,b,\lambda,x)\psi(q,\lambda,x)(1+x)^{-r}
\end{equation*}
and
\begin{equation*}
\frac{\partial w(q,b,\lambda_n(q,b))}{\partial q(x)}
	+ \dot{w}(q,b,\lambda_n(q,b))\frac{\partial\lambda_n(q,b)}{\partial q(x)}
	= 0,
\end{equation*}
the identity for the gradient is obtained after taking into account that
\begin{equation*}
\phi(q,b,\lambda_n(q,b),x) = \frac{\psi(q,\lambda_n(q,b),x)}{\psi(q,\lambda_n(q,b),0)},
\quad n\in\N,
\end{equation*}
which follows from the identity $\psi(q,z,x) = \psi(q,z,0)c(q,z,x) + \psi'(q,z,0)s(q,z,x)$.
\end{proof}

\begin{lemma}
\label{lem:denominator}
Assume $(q,b)\in\bsA_r\times\R$. Then,
\[
\norm{\psi(q,\lambda_n(q,b),\cdot)}^2_2
	= (\tfrac32\pi n)^{1/3} \left[1 + \bO(\omega_r(n))\right],
\]
uniformly on bounded subsets of $\bsA_r\times\R$.
\end{lemma}
\begin{proof}
In what follows we use the abbreviations $\psi_0(\lambda)\defeq\psi_0(\lambda,0)$ and
$\lambda_n\defeq\lambda_n(q,b)$. Clearly,
\begin{equation*}
\psi_0(\lambda_n)
	= \sqrt{\pi}\ai(-\lambda_n)
	= \lambda_n^{-1/4}
			\left[\cos\left(\tfrac23\lambda_n^{3/2}-\tfrac{\pi}{4}\right)
			+ \bO(\lambda_n^{-3/2})\right],
\end{equation*}
which, in conjunction with Lemma~\ref{lem:crude-asymp-eigenvalues}, implies
\begin{equation*}
\psi_0(\lambda_n)
	= (-1)^{n+1}(\tfrac32\pi n)^{-1/6}\left[1 + \bO(n^{-2/3+2\epsilon})\right]
\end{equation*}
so
\begin{equation*}
\lambda_n\psi_0(\lambda_n)^2
	= (-a_n')^{1/2}\left[1 + \bO(n^{-2/3+2\epsilon})\right]
	= (\tfrac32\pi n)^{1/3}\left[1 + \bO(n^{-2/3+2\epsilon})\right]
\end{equation*}
for arbitrary $\epsilon\in(0,1/6)$. Similarly,
\begin{equation*}
\psi_0'(\lambda_n) = \bO(n^{-1/6+\epsilon}).
\end{equation*}
Recalling that
\begin{equation*}
\norm{\psi(q,\lambda_n,\cdot)}^2_2
	= \dot{w}(q,b,\lambda_n)\psi(q,\lambda_n,0)
\end{equation*}
and then resorting to \eqref{eq:about-Xi}--\eqref{eq:about-Xi-dot} and \eqref{eq:w}, we can write
\begin{align*}
\norm{\psi(q,\lambda_n,\cdot)}^2_2
	=&\, \lambda_n\psi_0(\lambda_n)^2
		\left[1 + b \frac{\psi_0'(\lambda_n)}{\lambda_n\psi_0(\lambda_n)}
		+ \frac{\varXi(q,\lambda_n)}{\psi_0(\lambda_n)}
		+ \frac{\dot{\varXi}'(q,\lambda_n)}{\lambda_n\psi_0(\lambda_n)}
		- b \frac{\dot{\varXi}(q,\lambda_n)}{\lambda_n\psi_0(\lambda_n)}\right.
	\\
	 & + b \left. \frac{\psi_0'(\lambda_n)\varXi'(q,\lambda_n,0)}{\lambda_n\psi_0(\lambda_n)^2}
		+ \frac{\varXi(q,\lambda_n)\dot{\varXi}'(q,\lambda_n)}
					{\lambda_n\psi_0(\lambda_n)^2}
		- b\frac{\varXi(q,\lambda_n)\dot{\varXi}(q,\lambda_n)}
					{\lambda_n\psi_0(\lambda_n)^2}\right],
\end{align*}
where we assume $n$ large enough to avoid division by zero. Clearly,
\begin{equation*}
\frac{\psi_0'(\lambda_n)}{\lambda_n\psi_0(\lambda_n)}
	= \bO(n^{-2/3+\epsilon})
\implies
\frac{\psi_0'(\lambda_n)}{\lambda_n\psi_0(\lambda_n)}
	= \bO(\om_r(n)).
\end{equation*}
Also,
\begin{equation*}
\abs{\varXi(q,\lambda_n)}
	\le C\frac{\om(q,\lambda_n)}{\sigma(\lambda_n)}
\implies
\frac{\varXi(q,\lambda_n)}{\psi_0(\lambda_n)}
	= \bO(\omega_r(n)),
\end{equation*}
and
\begin{equation*}
\abs{\dot{\varXi}'(q,\lambda_n)}
	\le C\left[(1+\abs{\lambda_n}^{3/4})\bom(q,\lambda_n)
		+ \sigma(\lambda_n)\norm{q}_1\right]
\implies
\frac{\dot{\varXi}'(q,\lambda_n)}{\lambda_n\psi_0(\lambda_n)}
	= \bO(\omega_r(n)).
\end{equation*}
The remaining terms can be treated in a similar fashion, thus yielding the desired result.
\end{proof}

\begin{theorem}
\label{thm:eigenvalues}
Assume $(q,b)\in\bsA_r\times\R$. Then,
\begin{equation*}
\label{eq:eigenvalues}
\lambda_n(q,b)
	= - a_n' + \pi \frac{\int_0^\infty \ai^2(x+a_n')q(x)dx}{(-a_n')^{1/2}}
		- \frac{b}{a_n'}
		+ \bO(n^{-1/3}\omega_r^2(n)),
\end{equation*}
uniformly on bounded subsets of $\bsA_r\times\R$.
\end{theorem}

\begin{proof}
It suffices to prove the assertion on balls $\cV\subset\bsA_r$ centered a $q\equiv0$.
Let $\cU$ be the induced bounded subset in $\sA_r$, and let $\cJ$ be a bounded subset of $\R$.
In order to somewhat unclutter the exposition of this proof, we write
$\inner{\cdot}{\cdot}$ and $\norm{\cdot}$ instead of $\inner{\cdot}{\cdot}_2$ and $\norm{\cdot}_2$.

Given $q\in\cV$, we clearly have $tq\in\cV$, hence $tq\in\cU$, whenever $t\in[0,1]$.
Then, Proposition~\ref{lem:eingenvalue-is-real-analytic} leads to
\begin{equation*}
	\lambda_n(q,b) - \lambda_n(0,b)
	= \int_0^1 \frac{d}{dt}\lambda_n(tq,b)\, dt
	= \int_0^1 \inner{\eta^2_n(tq,b,\cdot)}{q}dt.
\end{equation*}
Resorting to \eqref{eq:about-Xi}, we can write
\begin{align}
	\lambda_n(q,b) - \lambda_n(0,b)
	=& \int_0^1 \frac{1}{\norm{\psi(tq,\lambda_n(tq,b),\cdot)}^2}
			\inner{\psi_0^2(\lambda_n(tq,b),\cdot)}{q} dt
				\nonumber
	\\
	 & + \int_0^1 \frac{1}
			{\norm{\psi(tq,\lambda_n(tq,b),\cdot)}^2}
			\inner{\varPsi_n(tq,\lambda_n(tq,b),\cdot)}{q} dt , \label{eq:decomposition}
\end{align}
where
\begin{equation*}
	\varPsi_n(tq,\lambda_n(tq,b),x) \defeq
		2\psi_0(\lambda_n(tq,b),x)
			\varXi(tq,\lambda_n(tq,b),x)
				+\varXi^2(tq,\lambda_n(tq,b),x).
\end{equation*}
Observe that, because of the first inequality in \eqref{eq:vbe-psi0-prime},
\begin{equation*}
	\abs{\inner{\psi_0^2(\lambda_n(tq,b),\cdot)}{q}}
		\le
			\pi C_0^2 \int_0^\infty \frac{g_A^2(x-\lambda_n(tq,b))}
				{\sigma(x-\lambda_n(tq,b))^2}\abs{q(x)}dx
					\le C\om(q,\lambda_n(tq,b)),
\end{equation*}
which  implies
\begin{equation*}
	\inner{\psi_0^2(\lambda_n(tq,b),\cdot)}{q}
		= \bO( \omega_r(n))
\end{equation*}
uniformly on $\cU\times\cJ$, hence on $\cV\times\cJ$. Then, Lemma \ref{lem:denominator} implies
\begin{equation}
	\label{eq:1st-term}
		\int_0^1 \frac{1}{\norm{\psi\left(tq,\lambda_n\left(tq,b\right),\cdot\right)}^2}
			\inner{\psi_0^2\left(\lambda_n\left(tq,b\right),\cdot\right)}{q} dt =
				\bO( n^{-1/3}\omega_r(n)),
\end{equation}
uniformly on $\cV\times\cJ$. On the other hand, due to \eqref{eq:about-Xi},
\begin{equation*}
	\abs{\varPsi_n(tq,\lambda_n\left(tq,b\right),x)}
		\le C \frac{\om(tq,\lambda_n(tq,b))}
				{( 1+\abs{x-\lambda_n\left(tq,b\right)})^{1/2}},
\end{equation*}
hence
\begin{equation*}
	\inner{\varPsi_n(tq,\lambda_n\left(tq,b\right),\cdot)}{q}
		= \bO(\omega_r^2(n)),
\end{equation*}
uniformly on $\cV\times\cJ$. Consequently, Lemma \ref{lem:denominator} implies,
\begin{equation}
	\label{eq:2nd-term}
		\int_0^1 \frac{1}{\norm{\psi\left(tq,\lambda_n\left(tq,b\right),\cdot\right)}^2}
		\inner{\varPsi_n(tq,\lambda_n\left(tq,b\right),\cdot)}{q} dt
			= \bO( n^{-1/3}\omega_r^2(n))
\end{equation}
uniformly on $\cV\times\cJ$.
Thus, in view of \eqref{eq:decomposition}--\eqref{eq:2nd-term}, so far we obtain
\begin{equation}
\label{eq:uniform-estimate}
	\lambda_n(q,b) - \lambda_n(0,b) = \bO(n^{-1/3}\omega_r(n)),
\end{equation}
uniformly on $\cV\times\cJ$.

Let us go back to \eqref{eq:1st-term} and write
\begin{equation}
\label{eq:decomposition-inner-products}
\inner{\psi_0^2(\lambda_n(tq,b),\cdot)}{q}
	= \inner{\psi_0^2(-a_n',\cdot)}{q} + I_1 + I_2,
\end{equation}
where
\begin{equation*}
I_1 = \inner{\psi_0^2(\lambda_n(tq,b),\cdot)-
		\psi_0^2(\lambda_n(0,b),\cdot)}{q} \quad\text{and}\quad
I_2 = \inner{\psi_0^2(\lambda_n(0,b),\cdot) -
		\psi_0^2(-a_n',\cdot)}{q}.
\end{equation*}
Let us abbreviate $\lambda_1\defeq \lambda_n(0,b)$ and $\lambda_2\defeq \lambda_n(tq,b)$.
We observe that
\begin{equation*}
\ai^2\left(x-\lambda_2\right) - \ai^2\left(x-\lambda_1\right)
	= - 2 \int_{\lambda_1}^{\lambda_2}\ai\left(x- u\right)\ai'\left(x- u\right) du.
\end{equation*}
Thus,
\begin{align*}
\abs{\ai^2\left(x-\lambda_2\right) - \ai^2\left(x-\lambda_1\right)}
	&\le 2 \int_{\lambda_1}^{\lambda_2}\abs{\ai\left(x- u\right)}\abs{\ai'\left(x- u)\right)} du
	 \le 2 C_0^2 \abs{\lambda_2 - \lambda_1},
\end{align*}
where we use \eqref{eq:vbe-psi0-prime} and the fact that $g_A(x)\le 1$ on the real line.
Therefore,
\begin{align*}
\abs{I_1}
	& \le \norm{q}_{\cA_r}
	\left(\int_0^\infty\abs{\psi_0^2(\lambda_n(tq,b),x) - \psi_0^2(\lambda_n(0,b),x)}^2 (1+x)^{-r}
	dx\right)^{1/2}
	\\
	&\le 2 C_0^2\pi \norm{q}_{\cA_r}
	\abs{\lambda_n(tq,b)-\lambda_n(0,b)}
	\left(\int_0^\infty(1+x)^{-r}dx\right)^{1/2},
\end{align*}
which in turn implies
\begin{equation}
\label{eq:1st-inner}
I_1 = \bO( n^{-1/3}\omega_r(n))
\end{equation}
uniformly on $\cV\times\cJ$, as a consequence of \eqref{eq:uniform-estimate}.
Resorting to an analogous reasoning, from Proposition \ref{prop:spectral-data-unperturbed-mixed-bc}
we obtain
\begin{equation}
\label{eq:2nd-inner}
I_2 = \bO(n^{-2/3})
\end{equation}
uniformly on $\cJ\times\cV$.

Summarizing, from \eqref{eq:decomposition-inner-products}--\eqref{eq:2nd-inner} we obtain
\begin{equation*}
\inner{\psi_0^2(\lambda_n(tq,b),\cdot)}{q}
	= \inner{\psi_0^2(-a_n',\cdot)}{q}
	+ \bO( n^{-1/3}\omega_r(n))
\end{equation*}
uniformly on $\cV\times\cJ$. Therefore, Lemma \ref{lem:denominator} yields,
\begin{equation*}
\label{eq:1st-term-refined}
\int_0^1 \frac{1}{\norm{\psi(tq,\lambda_n(tq,b),\cdot)}^2}
	\inner{\psi_0^2(\lambda_n(tq,b),\cdot)}{q}dt
	= \frac{\inner{\psi_0^2(-a_n',\cdot)}{q}}
	{(\frac{3}{2}\pi n)^{1/3}} + \bO( n^{-1/3}\omega^2_r(n)),
\end{equation*}
that is,
\begin{equation*}
\lambda_n(q,b)-\lambda_n(0,b)
	= \frac{\inner{\psi_0^2\left(-a_n',\cdot\right)}{q}}
		{(\frac{3}{2}\pi n)^{1/3}} + \bO( n^{-1/3}\omega^2_r(n)).
\end{equation*}
Then, the stated asymptotic expansion follows after applying \eqref{eq:zeros-of-ai-prime}
and Proposition \ref{prop:spectral-data-unperturbed-mixed-bc}.
\end{proof}


\section{The norming constants}
\label{sec:norming-constants}

For the sake of brevity, in what follows we use the notation
\[
\lambda_n	\defeq \lambda_n(q,b),\quad
\psi_n		\defeq \psi(q,\lambda_n,0),\quad
w_n			\defeq w(q,b,\lambda_n),\quad
\dot{w}_n			\defeq \dot{w}(q,b,\lambda_n)
\]
and so on.

\begin{proposition}
\label{lem:norming-constants-are-analytic}
Given $n\in\N$ and $b\in\R$, let $\kappa_n(\cdot,b):\sA_r\to\R$ be defined by
\eqref{eq:norming-constants}, that is,
\begin{equation*}
\kappa_n(q,b)
	= \log\left(\frac{\psi(q,\lambda_n(q,b),0)}{\dot{w}(q,b,\lambda_n(q,b))}\right).
\end{equation*}
Then, $\kappa_n(\cdot,b)$ is a real analytic map. Moreover,
\begin{equation*}
\frac{\partial\kappa_n}{\partial q(x)}
	= \left(\frac{1}{\psi_n}\frac{\partial\psi_n}{\partial q(x)}
		- \frac{1}{\dot{w}_n}\frac{\partial\dot{w}_n}{\partial q(x)}\right)
		+ \left(\frac{\dot{\psi}_n}{\psi_n} - \frac{\ddot{w}_n}{\dot{w}_n}\right)
		\frac{\partial\lambda_n}{\partial q(x)}.
\end{equation*}
\end{proposition}
\begin{proof}
Clearly, $\kappa_n$ is a composition of real analytic maps. Its gradient
with respect to $q$ follows after a straightforward computation.
\end{proof}

Let us define
\begin{equation}
\label{eq:A-B}
A_n(q,b,x)
	\defeq \frac{1}{\psi_n}\frac{\partial\psi_n}{\partial q(x)}
			- \frac{1}{\dot{w}_n}\frac{\partial\dot{w}_n}{\partial q(x)}
\quad\text{and}\quad
B_n(q,b)
	\defeq \frac{\dot{\psi}_n}{\psi_n} - \frac{\ddot{w}_n}{\dot{w}_n}.
\end{equation}
Clearly,
\begin{equation*}
\frac{\partial\kappa_n}{\partial q(x)}(q,b)
	= A_n(q,b,x) + B_n(q,b)\frac{\partial\lambda_n}{\partial q(x)}(q,b).
\end{equation*}

Let us look at the first definition in \eqref{eq:A-B}. We have
\begin{align*}
(1+x)^{r}A_n(q,b,x)
	=&\, \frac{1}{\psi_n}s(q,\lambda_n,x)\psi(q,\lambda_n,x)
	\\
	 & + \frac{1}{\dot{w}_n}
		\left[\dot{\phi}(q,b,\lambda_n,x)\psi(q,\lambda_n,x)
			+ \phi(q,b,\lambda_n,x)\dot{\psi}(q,\lambda_n,x)\right],
\end{align*}
where $\phi(q,b,z,x)$ is defined in \eqref{eq:the-other-function}.
It will be convenient to write $A_n(q,b,x)$ as follows,
\begin{align}
\label{eq:big-decomposition}
(1+x)^{r}A_n(q,b,x)
	=& \left(1+x\right)^{r}A_n^0(q,b,x)
	\nonumber
	\\
	 & + \left(\frac{\lambda_n\dot{W}_n}{\dot{\tau}_n\dot{w}_n}
			- \frac{\varXi_n}{\alpha_n\psi_n}\right)s_0(\lambda_n,x)\psi_0(\lambda_n,x)
	   - \frac{b\dot{W}_n}{\dot{\tau}_n\dot{w}_n}c_0(\lambda_n,x)\psi_0(\lambda_n,x)
	\nonumber
	\\
	 & + \frac{\dot{W}_n}{\dot{\tau}_n\dot{w}_n}
	 		\left[\phi'_0(b,\lambda_n,x)\psi_0(\lambda_n,x)
			+ \phi_0(b,\lambda_n,x)\psi'_0(\lambda_n,x)\right]
	\nonumber
	\\[1mm]
	 & + \frac{\varLambda_1(q,\lambda_n,x)}{\psi_n}
	   + \frac{\varLambda_2(q,b,\lambda_n,x)}{\dot{w}_n},
\end{align}
where
\begin{align*}
(1+x)^r A_n^0(q,b,x)
	\defeq& \left(\frac{\tau_n^\times}{\dot{\tau}_n} - \frac{\beta_n}{\alpha_n}\right)
			\psi_0^2(\lambda_n,x)
			- 2\frac{\beta_n' - b\beta_n}{\dot{\tau}_n}\psi_0(\lambda_n,x)\psi_0'(\lambda_n,x)
	\\
		  & + \frac{\alpha_n' - b\alpha_n}{\dot{\tau}_n}
		  	\left[\psi_0(\lambda_n,x)\theta_0'(\lambda_n,x)
		  	+ \psi_0'(\lambda_n,x)\theta_0(\lambda_n,x)\right],
\end{align*}
\begin{equation*}
\Lambda_1(q,z,x)
	\defeq s_0(z,x)\varXi(q,z,x) + \psi_0(z,x)\varUpsilon_s(q,z,x)
		+ \varUpsilon_s(q,z,x)\varXi(q,z,x)
\end{equation*}
and
\begin{align*}
\Lambda_2(q,b,z,x)
	\defeq&\,\dot{\phi}_0(b,z,x)\varXi(q,z,x) + \psi_0(z,x)\dot{\varPhi}(q,b,z,x)
			+ \dot{\varPhi}(q,b,z,x)\varXi(q,z,x)
	\\
		  & + \phi_0(b,z,x)\dot{\varXi}(q,z,x) - \psi_0'(z,x)\varPhi(q,b,z,x)
			+ \varPhi(q,b,z,x)\dot{\varXi}(q,z,x).
\end{align*}
The remaining new notation introduced in the preceding discussion is defined in
Remark~\ref{rem:alphas-and-betas}.
We note that the main contribution in \eqref{eq:big-decomposition} to the norming constant
is $A_n^0(q,b,x)$.

\begin{lemma}
\label{lem:A_n^0}
Assume $(q,b)\in\bsA_r\times\R$. Then, $A_n^0(q,b,\cdot)\in\sA_r$ and
\[
\int_0^1\inner{A_n^0(tq,b,\cdot)}{q}_{\sA_r} dt
	= \pi\frac{\int_0^\infty \ai^2(x+a_n')(q'(x)+bq(x)) dx}{(\tfrac32\pi n)^{1/3}}
		- \frac{q(0)}{(\tfrac32\pi n)^{2/3}}
		+ \bO(n^{-1/6}\om_r^2(n)),
\]
uniformly whenever $(q,q(0),b)$ belongs to bounded subsets of $\bsA_r\times\R\times\R$.
\end{lemma}
\begin{proof}
Let us suppose that $q$ lies in a ball $\cV\subset\bsA_r$ so $tq$
lies in a fixed bounded subset of $\sA_r$ for all $t\in[0,1]$. Also, assume $b$ lies in
some bounded interval $\cJ\subset\R$. We write $\inner{\cdot}{\cdot}$ to denote $\inner{\cdot}{\cdot}_2$.

Integration by parts yields
\begin{align}
\label{eq:long-summer}
\inner{A_n^0(tq,b,\cdot)}{q}_{\sA_r}
	=& \left(\frac{\tau_n^\times}{\dot{\tau}_n} - \frac{\beta_n}{\alpha_n}\right)
		\inner{\psi_0^2(\lambda_n(tq,b),\cdot)}{q}
		- \frac{\alpha_n}{\dot{\tau}_n} q(0)\nonumber
	\\[1mm]
	 & + \frac{\beta_n' - b\beta_n}{\dot{\tau}_n}\inner{\psi_0^2(\lambda_n(tq,b),\cdot)}{q'}
	   + \frac{\alpha_n' - b\alpha_n}{\dot{\tau}_n}
		 \inner{(\psi_0\theta_0)(\lambda_n(tq,b),\cdot)}{q'}.
\end{align}
Clearly,
\[
\abs{\inner{\psi_0^2(\lambda_n(tq,b),\cdot)}{q}}
	\le C \bom(q,\lambda_n(tq,b)),
\]
which in turn implies
\[
\inner{\psi_0^2(\lambda_n(tq,b),\cdot)}{q} = \bO(\om_r(n)),
\]
uniformly on $\cV\times\cJ$. Analogous asymptotic formulas are shown to hold for the last two terms
in \eqref{eq:long-summer}. Applying Lemma~\ref{lem:pieces}, we obtain
\[
\inner{A_n^0(tq,b,\cdot)}{q}_{\sA_r}
	= \frac{\inner{\psi_0^2(\lambda_n(tq,b),\cdot)}{q' + bq}}{(\tfrac32\pi n)^{1/3}}
		- \frac{\alpha_n}{\dot{\tau}_n} q(0)
		+ \bO(n^{-1/6}\om_r^2(n))
\]
uniformly on $\cV\times \cJ$. But we already know that
\begin{equation*}
\inner{\psi_0^2(\lambda_n(tq,b),\cdot)}{q}
	= \inner{\psi_0^2(-a_n',\cdot)}{q} + \bO(n^{-1/3}\om_r(n))
\end{equation*}
and the argument leading to this expression carries over for $q'$ so
\begin{equation*}
\inner{\psi_0^2(\lambda_n(tq,b),\cdot)}{q'}
	= \inner{\psi_0^2(-a_n',\cdot)}{q'} + \bO(n^{-1/3}\om_r(n)),
\end{equation*}
uniformly on $\cV\times \cJ$. Finally, from Lemma~\ref{lem:alpha-etcetera} and a comment in the proof of
Lemma~\ref{lem:pieces}, it follows that
\begin{equation*}
\frac{\alpha_n}{\dot{\tau}_n} = (\tfrac32\pi n)^{-2/3}\left[1 + \bO(\om_r^2(n))\right].
\end{equation*}
The proof is now complete.
\end{proof}

In what follows, we use the identity
\begin{equation*}
\dot{w}_n = \dot{\tau}_n + \dot{W}_n,\quad\text{where}\quad
\dot{W}_n \defeq \dot{\varXi}_n' - b \dot{\varXi}_n.
\end{equation*}

\begin{lemma}
\label{lem:ugly-thing}
Assume $(q,b)\in\bsA_r\times\R$. Then,
\begin{equation*}
\frac{\lambda_n\alpha_n\dot{W}_n}{\dot{\tau}_n\dot{w}_n}
	- \frac{\varXi_n}{\psi_n}
	= \bO(\omega_r^2(n)),
\end{equation*}
uniformly on bounded subsets of $\bsA_r\times\R$.
\end{lemma}
\begin{proof}
In view of Remark~\ref{rem:some-asymptotics},
\begin{equation*}
\dot{\varXi}'_n = \bO(n^{1/2}\omega_r(n))
\quad\text{and}\quad
\dot{\varXi}_n = \bO(n^{1/6}\omega_r(n))
\quad\text{so}\quad
\frac{\dot{W}_n}{\dot{\tau}_n} = \bO(\omega_r(n)).
\end{equation*}
Also,
\begin{equation*}
\frac{\alpha_n'}{\lambda_n\alpha_n}
	= \bO(n^{-1/3}\omega_r(n))
\quad\text{and, for later use,}\quad
\frac{\dot{W}_n}{\lambda_n\alpha_n}
	= \bO(\omega_r(n)).
\end{equation*}
As a consequence,
\begin{equation*}
\frac{\lambda_n\alpha_n\dot{W}_n}{\dot{\tau}_n\dot{w}_n}
	= \frac{\dot{W}_n}{\lambda_n\alpha_n}\left[1 + \bO(\omega_r(n))\right].
\end{equation*}
Similarly,
\begin{equation*}
\frac{\varXi_n}{\alpha_n} = \bO(\omega_r(n))
\quad\text{so}\quad
\frac{\varXi_n}{\psi_n} = \frac{\varXi_n}{\alpha_n} \left[1 + \bO(\omega_r(n))\right].
\end{equation*}
Thus, up to this point, we have
\begin{equation*}
\frac{\lambda_n\alpha_n\dot{W}_n}{\dot{\tau}_n\dot{w}_n}
	- \frac{\varXi_n}{\psi_n}
	= \frac{\dot{W}_n}{\lambda_n\alpha_n}
	- \frac{\varXi_n}{\alpha_n} + \bO(\omega_r^2(n)).
\end{equation*}

We now turn our attention to equations \eqref{eq:recall-1} and \eqref{eq:recall-2},
for they imply
\begin{equation*}
\frac{\dot{W}_n}{\lambda_n\alpha_n} - \frac{\varXi_n}{\alpha_n}
	= -\frac{q(0)}{\lambda_n}
	+ \frac{\dot{\psi'}_{1,n}^\text{res}}{\lambda_n\alpha_n}
	- \frac{b\dot{\varXi}_n}{\lambda_n\alpha_n}
	+ \frac{\dot{\varXi'}^{(2)}_n}{\lambda_n\alpha_n}
	- \frac{\varXi^{(2)}_n}{\alpha_n}.
\end{equation*}
Now, taking into account Remark~\ref{rem:As-and-Bs},
\begin{equation*}
\dot{\psi'}_{1,n}^\text{res}
	= \bO(n^{1/6}\omega_r(n)),
\quad\text{hence}\quad
	\frac{\dot{\psi'}_{1,n}^\text{res}}{\lambda_n\alpha_n}
	= \bO(n^{-1/3}\omega_r(n)).
\end{equation*}
Clearly,
\begin{equation*}
\frac{b\dot{\varXi}_n}{\lambda_n\alpha_n}
	= \bO(n^{-1/3}\omega_r(n)),
\quad
\frac{q(0)}{\lambda_n} = \bO(n^{-2/3})
\end{equation*}
and, in view of Remark~\ref{rem:more-of-the-same},
\begin{equation*}
\frac{\dot{\varXi'}^{(2)}_n}{\lambda_n\alpha_n}
	= \bO(\omega^2_r(n)),\quad
\frac{\varXi^{(2)}_n}{\alpha_n}
	= \bO(\omega^2_r(n)).
\end{equation*}
The assertion follows from these estimates.
\end{proof}

\begin{lemma}
\label{lem:the-2-boys}
Assume $(q,b,v)\in\bsA_r\times\R\times\sA_r$. Then,
\begin{equation*}
\inner{\varLambda_1(q,\lambda_n(q,b),\cdot)}{v}_2 = \bO(n^{-1/6}\om_r^2(n))
\quad\text{and}\quad
\inner{\varLambda_2(q,b,\lambda_n(q,b),\cdot)}{v}_2 = \bO(n^{1/2}\om_r^2(n)),
\end{equation*}
uniformly on bounded subsets of $\bsA_r\times\R\times\sA_r$.
\end{lemma}
\begin{proof}
As before, $\inner{\cdot}{\cdot}$ means $\inner{\cdot}{\cdot}_2$.
We have
\begin{equation*}
\inner{s_0\varXi}{v}
	= - \beta_n\inner{\psi_0\varXi}{v} + \alpha_n\inner{\theta_0\varXi}{v},
\end{equation*}
and note that
\begin{equation*}
\abs{\inner{\psi_0\varXi}{v}}
	\le C\om(q,\lambda_n)\om(v,\lambda_n)
\implies
\inner{\psi_0\varXi}{v} = \bO(\omega_r^2(n));
\end{equation*}
an analogous asymptotic expansion holds true for $\inner{\theta_0\varXi}{v}$. Taking into account
Lemma~\ref{lem:alpha-etcetera}, we obtain
\begin{equation}
\label{eq:1}
\inner{s_0\varXi}{v} = \bO(n^{-1/6}\omega_r^2(n)),
\end{equation}
uniformly on bounded subsets of $\bsA_r\times\R\times\sA_r$ (from this point on, we will not refer
to this remark explicitly). Next, we note that
\begin{equation*}
\ch(\lambda,x)g_A(x-\lambda)
	= g_B(-\lambda) g_A^2(x-\lambda) + g_A(-\lambda),
\end{equation*}
where $g_A(\lambda)\le 1$ on the whole real line and $g_B(-\lambda)\le 1$ whenever
$\lambda\ge 0$. As a consequence,
\begin{equation}
\label{eq:2}
\abs{\inner{\psi_0\varUpsilon_s}{v}}
	\le C\frac{\om(q,\lambda_n)\om(v,\lambda_n)}{\sigma(\lambda_n)}
\implies
\inner{\psi_0\varUpsilon_s}{v} = \bO(n^{-1/6}\omega_r^2(n)).
\end{equation}
Similarly,
\begin{equation}
\label{eq:3}
\inner{\varUpsilon_s\varXi}{v} = \bO(n^{-1/6}\omega_r^3(n)).
\end{equation}
Thus, the assertion concerning $\varLambda_1$ follows from \eqref{eq:1}--\eqref{eq:3}.

The proof of the remaining assertion runs through similar arguments. Resorting to
\eqref{eq:about-Xi} and \eqref{eq:estimate-dot-phi-0}, we obtain
\begin{equation*}
\abs{\inner{\dot{\phi}_0\varXi}{v}}
	\le C (1+\abs{b})\om(q,\lambda_n)
		\left(\frac{\abs{\lambda_n}}{\sigma(\lambda_n)}\om(v,\lambda_n) +
		\sigma(\lambda_n)\norm{v}_{\sA_r}\right).
\end{equation*}
That is,
\begin{equation*}
\inner{\dot{\phi}_0\varXi}{v} = \bO(n^{1/2}\omega_r^2(n)).
\end{equation*}
Also,
\begin{equation*}
\inner{\phi_0\dot{\varXi}}{v}
	= (\beta_n' - b \beta_n)\inner{\psi_0\dot{\varXi}}{v}
		- (\alpha_n' - b \alpha_n)\inner{\theta_0\dot{\varXi}}{v},
\end{equation*}
where
\begin{equation*}
\inner{\psi_0\dot{\varXi}}{v} = \bO(\omega_r^2(n))
\quad\text{and}\quad
\inner{\theta_0\dot{\varXi}}{v} = \bO(\omega_r(n))
\implies
\inner{\phi_0\dot{\varXi}}{v} = \bO(n^{1/6}\omega_r(n)).
\end{equation*}
Furthermore,
\begin{equation*}
\inner{\psi_0'\varPhi}{v}
	= \inner{\psi_0'\varUpsilon_c}{v} + b \inner{\psi_0'\varUpsilon_s}{v},
\end{equation*}
where an argument like the one leading to \eqref{eq:2} implies
\begin{equation*}
\inner{\psi_0'\varUpsilon_c}{v} = \bO(n^{1/6}\omega_r(n))
\quad\text{and}\quad
\inner{\psi_0'\varUpsilon_s}{v} = \bO(n^{-1/6}\omega_r(n))
\quad\text{so}\quad
\inner{\psi_0'\varPhi}{v}
	= \bO(n^{1/6}\omega_r(n)).
\end{equation*}
Similarly,
\begin{equation*}
\inner{\psi_0\dot{\varPhi}}{v}
	= \inner{\psi_0\dot{\varUpsilon}_c}{v} + b \inner{\psi_0\dot{\varUpsilon}_s}{v}
\end{equation*}
which, in view of Lemma~\ref{lem:about-c} and Lemma~\ref{lem:about-s}, yieds
\begin{equation*}
\inner{\psi_0\dot{\varPhi}}{v} = \bO(n^{1/2}\omega^2_r(n)).
\end{equation*}
Finally, one can verify that
\begin{equation*}
\inner{\dot{\varPhi}\varXi}{v}
	= \bO(n^{1/6}\omega_r^2(n))
\quad\text{and}\quad
\inner{\varPhi\dot{\varXi}}{v}
	= \bO(n^{1/6}\omega_r^2(n))
\end{equation*}
using analogous arguments.
\end{proof}

The proof of following result makes use of a decomposition of the function $\phi$ and its
partial derivatives that are discussed in Remark~\ref{rem:the-end}.

\begin{lemma}
\label{lem:Lambda_n}
Define $\varLambda_n:\bsA_r\times\R\to\sA_r$ according to the rule
\begin{equation*}
(1+x)^{r}\varLambda_n(q,b,x)
	\defeq \frac{\varLambda_1(q,\lambda_n,x)}{\psi_n}
			+ \frac{\varLambda_2(q,b,\lambda_n,x)}{\dot{w}_n}.
\end{equation*}
Then,
\begin{equation*}
\inner{\varLambda_n(q,b,\cdot)}{q}_{\sA_r}
	= \inner{(1+\cdot)^{r}\varLambda_n(q,b,\cdot)}{q}_2
	= \bO(\om_r^3(n)),
\end{equation*}
uniformly for $(q,q(0),b)$ on bounded subsets of $\bsA_r\times\R\times\R$.
\end{lemma}
\begin{proof}
In what follows $\inner{\cdot}{\cdot}$ means $\inner{\cdot}{\cdot}_{2}$. A computation yields
\begin{align}
\frac{\varLambda_1}{\psi_n} + \frac{\varLambda_2}{\dot{w}_n}
	&= \frac{\varLambda_1}{\alpha_n} + \frac{\varLambda_2}{\dot{\tau}_n}
		- \frac{\varXi_n}{\alpha_n\psi_n}\varLambda_1
		- \frac{\dot{W}_n}{\dot{\tau}_n\dot{w}_n}\varLambda_2\nonumber
	\\
	&= \frac{\varLambda_1}{\alpha_n} + \frac{\varLambda_2}{\lambda_n\alpha_n}
		- \frac{\varXi_n}{\alpha_n\psi_n}\varLambda_1
		- \frac{\dot{W}_n}{\dot{\tau}_n\dot{w}_n}\varLambda_2
		- \frac{b\alpha'_n}{\lambda_n\alpha_n\dot{\tau}_n}\varLambda_2.\label{eq:full-expansion}
\end{align}
Since
\begin{equation*}
\frac{\varXi_n}{\alpha_n\psi_n}
	= \bO(n^{1/6}\om_r(n)),
\quad
\frac{\dot{W}_n}{\dot{\tau}_n\dot{w}_n}
	= \bO(n^{-1/2}\om_r(n))
\quad\text{and}\quad
\frac{b\alpha'_n}{\lambda_n\alpha_n\dot{\tau}_n}
	= \bO(n^{-5/6}\om_r(n)),
\end{equation*}
Lemma~\ref{lem:the-2-boys} implies
\begin{equation*}
\frac{\varXi_n}{\lambda_n\alpha_n}\inner{\varLambda_1}{q}
	= \bO(\om_r^3(n)),
\quad
\frac{\dot{W}_n}{\dot{\tau}_n\dot{w}_n}\inner{\varLambda_2}{q}
	= \bO(\om_r^3(n)),
\end{equation*}
and
\begin{equation*}
\frac{b\alpha'_n}{\lambda_n\alpha_n\dot{\tau}_n}\inner{\varLambda_2}{q}
	= \bO(n^{-1/3}\om_r^3(n)),
\end{equation*}
uniformly for $(q,b)$ on bounded subsets of $\bsA_r\times\R$.

In dealing with the two remaining terms in \eqref{eq:full-expansion}, we use
the identities \eqref{eq:s-dot-c-dot} on $\dot{\phi}_0 = \dot{c}_0 + b\dot{s}_0$. Thus, we obtain
\begin{equation}
\label{eq:pain-in-the-ass}
\lambda_n\varLambda_1 + \varLambda_2
	= b c_0\varXi + \phi_0\dot{\varXi} - \phi_0'\varXi
		+ \lambda_n\psi_0\varUpsilon_s + \psi_0\dot{\varPhi} - \psi_0'\varPhi
		+ \lambda_n\varUpsilon_s\varXi + \dot{\varPhi}\varXi + \varPhi\dot{\varXi}.
\end{equation}
Then,
\begin{equation*}
\frac{1}{\lambda_n\alpha_n}\inner{c_0\varXi}{q}
	= \frac{\beta_n'}{\lambda_n\alpha_n}\inner{\psi_0\varXi}{q}
		- \frac{\alpha_n'}{\lambda_n\alpha_n}\inner{\theta_0\varXi}{q}
	= \bO(n^{-1/3}\om_r^2(n)).
\end{equation*}
Furthermore,
\begin{equation}
\label{eq:two-more}
\inner{\phi_0\dot{\varXi} - \phi_0'\varXi}{q}
	= - \inner{(\phi_0\psi_1)'}{q} + \inner{\phi_0\dot{\psi}_1^\text{res}}{q}
		+ \inner{\phi_0\dot{\varXi}^{(2)}}{q} - \inner{\phi'_0\varXi^{(2)}}{q}.
\end{equation}
An integration by parts yields
\begin{align*}
- \inner{(\phi_0\psi_1)'}{q}
	&= \phi_0(\lambda_n,0)\psi_1(\lambda_n,0)q(0) + \inner{\phi_0\psi_1}{q'}
	\\[1mm]
	&= \psi_1(\lambda_n,0)q(0)
		+ (\beta_n' - b\beta_n)\inner{\psi_0\psi_1}{q'}
		+ (\alpha_n' - b\alpha_n)\inner{\theta_0\psi_1}{q'}.
\end{align*}
But
\begin{equation*}
\abs{\inner{\psi_0\psi_1}{q'}}
	\le C \om(q,\lambda_n)\om(q',\lambda_n)
\quad\text{and}\quad
\abs{\inner{\theta_0\psi_1}{q'}}
	\le C \om(q,\lambda_n)\om(q',\lambda_n)
\end{equation*}
so, as a consequence of Lemma~\ref{lem:pieces} and the fact that $\psi_{1,n} = \bO(n^{-1/6}\om_r(n))$,
we obtain
\begin{equation*}
\frac{1}{\lambda_n\alpha_n}\inner{(\phi_0\psi_1)'}{q}
	= \bO(n^{-1/3}\om_r^2(n)).
\end{equation*}
Analogous results holds for the remaining terms in \eqref{eq:two-more}. Therefore,
\begin{equation*}
\frac{1}{\lambda_n\alpha_n}\inner{\phi_0\dot{\varXi} - \phi_0'\varXi}{q}
	= \bO(n^{-1/3}\om_r^2(n)),
\end{equation*}
uniformly for $(q,q(0),b)$ on bounded subsets of $\bsA_r\times\R\times\R$.
Let us continue with the next three terms in \eqref{eq:pain-in-the-ass}. Some cancellation
and regrouping leads to
\begin{equation}
\label{eq:intermediate-step}
\lambda_n\psi_0\varUpsilon_s + \psi_0\dot{\varPhi} - \psi_0'\varPhi
	= - (\psi_0\phi_1)' + b\psi_0 c_1 + \psi_0\dot{\phi}_1^\text{res}
		+ \lambda_n\psi_0\varUpsilon_s^{(2)} + \psi_0\dot{\varPhi}^{(2)}
		+ \psi_0'\varPhi^{(2)}.
\end{equation}
An integration by parts yields
\begin{equation*}
	\inner{(\psi_0\phi_1)'}{q} = - \inner{\psi_0\phi_1}{q'}
\end{equation*}
(see Remark~\ref{rem:the-end}). Thus, recalling \eqref{eq:vbe-psi0-prime} and \eqref{eq:phi-1},
we obtain
\begin{equation*}
\abs{\inner{\psi_0\phi_1}{q'}} \le
	C(1+\abs{b}) \sigma(\lambda_n) \om(q,\lambda_n)\om(q',\lambda_n)
	\implies
\frac{1}{\lambda_n \alpha_n}\inner{(\psi_0\phi_1)'}{q}
	= \bO(n^{-1/3}\om_r^2(n))
\end{equation*}
(from here up to the end of the paragraph all the asymptotics are uniform for $(q,b)$ on bounded subsets of
$\bsA_r\times\R$).
An analogous procedure applies for the next two terms in \eqref{eq:intermediate-step}.
As for the fourth term in \eqref{eq:intermediate-step},
\begin{equation*}
\abs{\inner{\psi_0\varUpsilon^{(2)}_s}{q}}
	\le C \frac{\om^3(q,\lambda_n)}{\sigma(\lambda_n)}
	\implies
	\frac{1}{\alpha_n}\inner{\psi_0\varUpsilon^{(2)}_s}{q}
	= \bO(\om_r^3(n)),
\end{equation*}
this due to the inequality
\begin{equation*}
\abs{\varUpsilon^{(2)}_s(q,\lambda,x)}
	\le C \om^2(q,\lambda)
			\frac{\ch(\lambda,x)}{\sigma(\lambda)\sigma(x-\lambda)},\quad \lambda\ge 0.
\end{equation*}
Next, due to \eqref{eq:Phi-dot-2},
\begin{equation*}
\frac{1}{\lambda_n\alpha_n}\inner{\psi_0\dot{\varPhi}^{(2)}}{q} = \bO(\om_r^3(n));
\end{equation*}
a similar bound holds for the last term in \eqref{eq:intermediate-step}.

Finally, for the last three terms in \eqref{eq:pain-in-the-ass} we have
\begin{equation*}
\frac{1}{\alpha_n}\inner{\varUpsilon_s \varXi}{q}
	= \bO(\om_r^3(n)),
\quad
\frac{1}{\lambda_n \alpha_n}\inner{\dot{\varPhi}\varXi}{q}
	= \bO(\om_r^3(n)),
\quad
\frac{1}{\lambda_n \alpha_n}\inner{\varPhi\dot{\varXi}}{q}
	= \bO(n^{-1/3} \om_r^2(n)),
\end{equation*}
uniformly for $(q,q(0),b)$ on bounded subsets of $\bsA_r\times\R\times\R$.
\end{proof}

\begin{lemma}
\label{lem:asymptotics-for-B}
Assume $(q,b)\in\bsA_r\times\R$. Then,
\begin{equation*}
\frac{\dot{\psi}_n}{\psi_n} - \frac{\ddot{w}_n}{\dot{w}_n}
	= - b + \bO(n^{1/3}\om^2_r(n)),
\end{equation*}
uniformly on bounded subsets of $\bsA_r\times\R$.
\end{lemma}
\begin{proof}
In terms of the notation already introduced, we can write
\begin{align*}
\dot{\psi}_n \dot{w}_n - \psi_n\ddot{w}_n
	=&	\left(\lambda_n\alpha_n\tau_n - \alpha_n'\dot{\tau}_n\right)
	\\
	 &	+ \left(\lambda_n\tau_n\varXi_n + \dot{\tau}_n\dot{\varXi}_n\right)
	 	- \left(\alpha_n'\dot{W}_n  + \alpha_n\ddot{W}_n\right)
		+ \left(\dot{\varXi}_n\dot{W}_n - \varXi_n\ddot{W}_n\right).
\end{align*}
Let us work out each of the terms above. Thus,
\begin{align*}
\text{1st term}
	&= - b \left(\lambda_n\alpha_n^2 + {\alpha_n'}^2\right)
	\\
	&= - b\, \lambda_n\alpha_n^2 \left(1 + \frac{{\alpha_n'}^2}{\lambda_n\alpha_n^2}\right)
	 = \lambda_n\alpha^2_n\left[-b + \bO(\om_r^2(n))\right].
\end{align*}
Next,
\begin{equation*}
	\text{2nd term}
	= \lambda_n\left(\alpha_n'\varXi_n + \alpha_n\dot{\varXi}_n\right)
	- b\left(\lambda_n\alpha_n\varXi_n - \alpha_n'\dot{\varXi}_n\right).
\end{equation*}
Recalling the decomposition discussed in Remark~\ref{rem:refinement}, in conjunction
with Remark~\ref{rem:As-and-Bs}, we obtain
\begin{equation*}
	\alpha_n'\varXi_n + \alpha_n\dot{\varXi}_n
	= (\alpha_n\beta_n' - \alpha_n'\beta_n) P_n
	+ \alpha_n \dot{\psi}_{1,n}^\text{res} + \alpha'_n \varXi_n^{(2)} + \alpha_n \dot{\varXi}_n^{(2)}.
\end{equation*}
Since $\alpha_n\beta_n' - \alpha_n'\beta_n = W(\psi_0(\lambda_n),\theta_0(\lambda_n)) = 1$
and $\dot{\psi}_{1,n}^\text{res} = \bO(n^{-1/6}\omega_r(n))$, and taking into account
Remark~\ref{rem:more-of-the-same} and Lemma~\ref{lem:alpha-etcetera}, it follows that
\begin{equation*}
\alpha_n'\varXi_n + \alpha_n\dot{\varXi}_n
	= P_n + \bO(\om_r^2(n))
\end{equation*}
On the other hand,
\begin{equation*}
\varXi_n - \frac{\alpha'_n}{\lambda_n\alpha_n}\dot{\varXi}_n = \bO(n^{-1/6}\om_r(n)).
\end{equation*}
Therefore, allowing some degree of informality,
\begin{equation*}
\text{2nd term}
	= \lambda_n P_n + \lambda_n\alpha_n^2 \bO(n^{1/3}\om_r^2(n)).
\end{equation*}
Now,
\begin{equation*}
\text{3rd term}
	= \left(\alpha_n'\dot{\varXi'}_n + \alpha_n\ddot{\varXi'}_n\right)
	- b \left(\alpha_n'\dot{\varXi}_n + \alpha_n\ddot{\varXi}_n\right).
\end{equation*}
Since
\begin{align*}
\alpha_n'\dot{\varXi'}_n + \alpha_n\ddot{\varXi'}_n
	=&\ \lambda_n P_n -\alpha_n\beta_n P_n + \alpha_n^2 Q_n
	  - (\alpha_n'\beta_n + \lambda_n\alpha_n\beta_n)P_n'
	\\
	 & + (\alpha_n\alpha_n' + \lambda_n\alpha_n^2)Q_n'
       + \alpha_n\partial_z\dot{\psi'}^\text{res}_{1,n}
	   + \alpha_n'\dot{\varXi}_n'^{(2)} + \alpha_n\ddot{\varXi}_n'^{(2)},
\end{align*}
it follows that (here we make use of Remark~\ref{rem:derivatives-psi-1})
\begin{equation*}
\alpha_n'\dot{\varXi'}_n + \alpha_n\ddot{\varXi'}_n
	= \lambda_n P_n + \lambda_n\alpha_n^2\bO(n^{1/3}\om_r^2(n)).
\end{equation*}
Analogously,
\begin{align*}
\alpha_n'\dot{\varXi}_n + \alpha_n\ddot{\varXi}_n
	=&\ P'_n + (\alpha'_n\beta_n + \lambda_n\alpha_n\beta_n)P_n
	\\
	 & - ({\alpha'_n}^2 + \lambda_n\alpha_n^2)Q_n
	   + \alpha_n\partial_z\dot{\psi}^\text{res}_{1,n}
	   + \alpha_n'\dot{\varXi}_n^{(2)} + \alpha_n\ddot{\varXi}_n^{(2)},
\end{align*}
whence
\begin{equation*}
\alpha_n'\dot{\varXi}_n + \alpha_n\ddot{\varXi}_n
	= \lambda_n\alpha_n^2\bO(\om_r(n)).
\end{equation*}
That is,
\begin{equation*}
\text{3rd term}
	= \lambda_n P_n + \lambda_n\alpha_n^2\bO(n^{1/3}\om_r^2(n)).
\end{equation*}
Finally, let us consider the last term. First we recall that
\begin{equation*}
\frac{\varXi_n}{\alpha_n} = \bO(\om_r(n)),\quad
\frac{\dot{\varXi}_n}{\alpha_n} = \bO(n^{1/3}\om_r(n))\quad\text{and}\quad
\frac{\dot{W}_n}{\dot{\tau}_n} = \bO(\om_r(n)).
\end{equation*}
Moreover, as a consequence of Lemma~\ref{lem:z-derivatives},
\begin{equation*}
\ddot{\varXi}_n = \bO(n^{1/2}\omega_r(n))
\quad\text{and}\quad
\ddot{\varXi'}_n = \bO(n^{5/6}\omega_r(n))
\end{equation*}
so
\begin{equation*}
\ddot{W}_n = \ddot{\varXi'}_n - b \ddot{\varXi}_n = \bO(n^{5/6}\omega_r(n))
\implies
\frac{\ddot{W}_n}{\dot{\tau}_n} = \bO(n^{1/3}\om_r(n)).
\end{equation*}
Therefore,
\begin{equation*}
\text{4th term}
	= \alpha_n\dot{\tau}_n\left(\frac{\dot{\varXi}_n}{\alpha_n}\frac{\dot{W}_n}{\dot{\tau}_n}
		-  \frac{\varXi_n}{\alpha_n}\frac{\ddot{W}_n}{\dot{\tau}_n}\right)
	= \alpha_n\dot{\tau}_n \bO(n^{1/3}\om^2_r(n)).
\end{equation*}
Finally,
\begin{equation*}
\psi_n\dot{w}_n
	= \alpha_n\dot{\tau}_n \left(1 + \frac{\varXi_n}{\alpha_n} + \frac{\dot{W}_n}{\dot{\tau}_n}
		+ \frac{\varXi_n}{\alpha_n}\frac{\dot{W}_n}{\dot{\tau}_n}\right)
	= \alpha_n\dot{\tau}_n\Bigl[1 + \bO(\om_r(n))\Bigr]
	= \alpha_n^2\lambda_n\Bigl[1 + \bO(\om_r(n))\Bigr].
\end{equation*}
The assertion follows from putting all these pieces together.
\end{proof}

\begin{theorem}
\label{thm:norming-constants}
Assume $(q,b)\in\bsA_r\times\R$. Then,
\begin{equation*}
\kappa_n(q,b)
	= - \log(-a_n')
		- 2\pi \frac{\int_0^\infty \ai(x+a_n')\ai'(x+a_n')q(x)dx}{(-a_n')^{1/2}}
		+ \frac{q(0) + b^2}{a_n'}
		+ \bO(n^{-1/6}\omega_r^2(n)),
\end{equation*}
where this expansion is uniform for $(q,q(0),b)$ on bounded subsets of $\bsA_r\times\R\times\R$.
\end{theorem}
\begin{proof}
Clearly,
\begin{align*}
\kappa_n(q,b) - \kappa_n(0,b)
	&= \int_0^1 \inner{\frac{\partial\kappa_n}{\partial(tq)}}{q}_{\sA_r} dt
	\\
	&= \int_{0}^{1} \inner{A_n(tq,b,\cdot)}{q}_{\sA_r} dt
		+ \int_{0}^{1} B_n(tq,b) \inner{\frac{\partial\lambda_n}{\partial(tq)}}{q}_{\sA_r} dt.
\end{align*}
In view of \eqref{eq:big-decomposition}, we have
\begin{equation*}
A_n(q,b,x) = A_n^0(q,b,x) + \varDelta_n(q,b,x),
\end{equation*}
where
\begin{align*}
(1+x)^r\varDelta_n(q,b,x)
	\defeq& \left(\frac{\lambda_n\dot{W}_n}{\dot{\tau}_n\dot{w}_n}
		- \frac{\varXi_n}{\alpha_n\psi_n}\right)s_0(\lambda_n,x)\psi_0(\lambda_n,x)
		- \frac{b\dot{W}_n}{\dot{\tau}_n\dot{w}_n}c_0(\lambda_n,x)\psi_0(\lambda_n,x)
	\\
	 & + \frac{\dot{W}_n}{\dot{\tau}_n\dot{w}_n}
		\left[\phi'_0(b,\lambda_n,x)\psi_0(\lambda_n,x)
		+ \phi_0(b,\lambda_n,x)\psi'_0(\lambda_n,x)\right]
	\\[1mm]
	 & + (1+x)^r\varLambda_n(q,b,x);
\end{align*}
the last term is defined in Lemma~\ref{lem:Lambda_n}. Clearly, Lemma~\ref{lem:A_n^0} yields
\begin{equation*}
\int_{0}^{1} \inner{A_n^0(tq,b,\cdot)}{q}_{\sA_r} dt
	= \frac{\inner{\psi_0^2(-a_n',\cdot)}{q' + b q}_2}{(-a_n')^{1/2}}
		+ \frac{q(0)}{a_n'}
		+ \bO(n^{-1/6}\om_r^2(n)),
\end{equation*}
while Lemma~\ref{lem:asymptotics-for-B} implies
\begin{equation*}
\int_{0}^{1} B_n(tq,b) \inner{\frac{\partial\lambda_n}{\partial(tq)}}{q}_{\sA_r} dt
	= - \,b \frac{\inner{\psi_0^2(-a_n',\cdot)}{q}_2}{(-a_n')^{1/2}} + \bO(n^{-1/3}\om^2_r(n))
\end{equation*}
so, up to this point,
\begin{equation*}
\kappa_n(q,b) - \kappa_n(0,b)
	= \frac{\inner{\psi_0^2(-a_n',\cdot)}{q'}_2}{(-a_n')^{1/2}}
		+ \frac{q(0)}{a_n'}
		+ \int_{0}^{1} \inner{\varDelta_n(tq,b,\cdot)}{q}_{\sA_r} dt
		+ \bO(n^{-1/6}\om_r^2(n)).
\end{equation*}

We claim that
\begin{equation*}
\int_{0}^{1} \inner{\varDelta_n(tq,b,\cdot)}{q}_{\sA_r} dt
	= \bO(n^{-1/2}\om_r(n)).
\end{equation*}
To begin with,
\begin{align*}
\inner{\varDelta_n(tq,b,\cdot)}{q}_{\sA_r}
	=& \left(\frac{\lambda_n\dot{W}_n}{\dot{\tau}_n\dot{w}_n}
		- \frac{\varXi_n}{\alpha_n\psi_n}\right)\inner{(s_0\psi_0)(\lambda_n,\cdot)}{q}_2
		- \frac{b\dot{W}_n}{\dot{\tau}_n\dot{w}_n}\inner{(c_0\psi_0)(\lambda_n,\cdot)}{q}_2
	\\
	 & + \frac{\dot{W}_n}{\dot{\tau}_n\dot{w}_n}\inner{(\phi_0\psi_0)'(\lambda_n,\cdot)}{q}_2
			+ \inner{\varLambda_n(q,b,\cdot)}{q}_{\sA_r}.
\end{align*}
Moreover,
\begin{equation*}
\inner{(s_0\psi_0)(\lambda_n,\cdot)}{q}_2
	= - \beta_n \inner{\psi_0^2(\lambda_n,\cdot)}{q}_2
		+ \alpha_n \inner{(\theta_0\psi_0)(\lambda_n,\cdot)}{q}_2
	= \bO(n^{-1/6}\om_r(n))
\end{equation*}
and, analogously,
\begin{equation*}
\inner{(c_0\psi_0)(\lambda_n,\cdot)}{q}_2 = \bO(n^{1/6}\om_r(n)).
\end{equation*}
Moving forward, an integration by parts yields
\begin{align*}
\inner{(\phi_0\psi_0)'(\lambda_n,\cdot)}{q}_2
	=& - q(0) - (\beta_n' - b \beta_n)\inner{\psi_0^2(\lambda_n,\cdot)}{q'}_2
		+ (\alpha_n' - b \alpha_n)\inner{(\psi_0\theta_0)(\lambda_n,\cdot)}{q'}_2
	\\
	=& - q(0) + \bO(n^{1/6}\om_r(n)).
\end{align*}
Thus, so far we have
\begin{equation*}
\inner{\varDelta_n(tq,b,\cdot)}{q}_{\sA_r}
	= \bO(n^{-1/2}\om_r(n)) + \inner{\varLambda_n(q,b,\cdot)}{q}_{\sA_r},
\end{equation*}
once we recall Lemma~\ref{lem:ugly-thing}. Thus, the assertion follows since the last term above is $\bO(\om_r^3(n))$ according to Lemma~\ref{lem:Lambda_n}.
\end{proof}


\appendix

\section{Appendix}
\label{sec:appendix}

\begin{remark}
\label{rem:alphas-and-betas}
Let us define
\begin{equation*}
\alpha_n	\defeq \psi_0(\lambda_n,0),\quad
\alpha_n'	\defeq \psi_0'(\lambda_n,0),\quad
\beta_n		\defeq \theta_0(\lambda_n,0)\quad\text{and}\quad
\beta_n'	\defeq \theta_0'(\lambda_n,0).
\end{equation*}
Also,
\begin{equation*}
\dot{\tau}_n
	\defeq \dot{w}(0,b,\lambda_n)
	= \lambda_n \alpha_n + b \alpha_n'\quad\text{and}\quad
\tau_n^\times
	\defeq \lambda_n \beta_n + b \beta_n'.
\end{equation*}
We observe that Theorem~\ref{thm:eigenvalues} implies
\begin{equation}
	\label{eq:uniform-estimate-eigenvalues}
	\lambda_n = \lambda_n(q,b) = - a_n' + \bO(n^{-1/3}\omega_r(n)),
\end{equation}
uniformly on bounded sets of $\bsA_r\times\R$.
\end{remark}

\begin{lemma}
\label{lem:alpha-etcetera}
The asymptotic formulas
\begin{align}
\alpha_n  &= (-1)^{n+1}(\tfrac32\pi n)^{-1/6}
			\left[1 + \bO(\omega^2_r(n))\right],\label{eq:alpha}
\\
\alpha_n' &= \bO(n^{1/6}\omega_r(n)), \nonumber
\\
\beta_n   &= \bO(n^{-1/6}\omega_r(n)) \quad\text{and}\nonumber
\\
\beta_n'  &= (-1)^{n+1}(\tfrac32\pi n)^{1/6}
			\left[1 + \bO(\omega^2_r(n))\right] \label{eq:beta-prime}
\end{align}
hold uniformly on bounded subsets of $\bsA_r\times\R$.
\end{lemma}

\begin{proof}
We write the proof for $\alpha_n$ and $\beta_n'$ since the remaining formulas
can be proven following the same reasoning.
In view of \eqref{eq:zeros-of-ai-prime} and \eqref{eq:uniform-estimate-eigenvalues}, we see that
\begin{equation}
\label{eq:powers-of-eigenvalues}
\lambda^{3/2}_n= \tfrac32 \pi (n-\tfrac34) + \bO(\omega_r(n)),
\quad
\lambda^{1/4}_n = (\tfrac32 \pi n)^{1/6}
		\left[1+\bO(n^{-1})\right].
\end{equation}
As a consequence,
\begin{equation}
\label{eq:cosine-powers-eigenvalues}
\cos(\tfrac23 \lambda^{3/2}_n - \tfrac{\pi}{4})
	= \cos\left(\pi(n-1)+\bO(\omega_r(n))\right)
	= (-1)^{n+1}+\bO(\omega^2_r(n)).
\end{equation}
Since
\begin{equation*}
\alpha_n
	= \sqrt{\pi} \ai(-\lambda_n)
	= \frac{1}{\lambda^{1/4}_n}
		\left[\cos(\tfrac{2}{3}\lambda^{3/2}_n - \tfrac{\pi}{4})
		+ \bO(\lambda_n^{-3/2})\right],
\end{equation*}
we obtain \eqref{eq:alpha}. Similarly,
\begin{equation}
\label{eq:asymptotic-bi}
\beta_n'
	= \sqrt{\pi} \bi'(-\lambda_n)
	= \lambda^{1/4}_n
		\left[ \cos( \tfrac{2}{3}\lambda^{3/2}_n - \tfrac{\pi}{4})
		+ \bO(\lambda_n^{-3/2}) \right].
\end{equation}
Thus, \eqref{eq:powers-of-eigenvalues}--\eqref{eq:asymptotic-bi}
yields \eqref{eq:beta-prime}.
\end{proof}

\begin{lemma}
\label{lem:pieces}
The asymptotic formulas
\begin{align*}
\frac{\alpha_n' - b\alpha_n}{\dot{\tau}_n}
	&= \bO(n^{-1/3}\omega_r(n)),
	\\
\frac{\beta_n' - b\beta_n}{\dot{\tau}_n}
	&= (\tfrac32\pi n)^{-1/3}
		\left[1 + \bO(n^{1/6}\omega_r(n))\right]
\quad\text{and}
	\\
\frac{\tau_n^\times}{\dot{\tau}_n} - \frac{\beta_n}{\alpha_n}
	&= b(\tfrac32\pi n)^{-1/3}\left[1 + \bO(\omega_r^2(n))\right]
\end{align*}
hold uniformly on bounded subsets of $\bsA_r\times\R$.
\end{lemma}
\begin{proof}
The first two asymptotic expansions are direct consequences of the previous lemma, along
with
\begin{equation*}
\dot{\tau}_n = (-1)^{n+1}(\tfrac32\pi n)^{1/2}
			\left[1 + \bO(\omega^2_r(n))\right],
\end{equation*}
which holds locally uniformly on $\bsA_r\times\R$. To obtain the remaining one, we first note
that
\begin{equation*}
\frac{\alpha_n'}{\lambda_n\alpha_n}
	= \bO(n^{-1/3}\omega_r(n))\quad\text{and}\quad
\frac{\beta_n'}{\lambda_n\alpha_n}
	= (\tfrac32\pi n)^{-1/3}\left[1 + \bO(\omega_r^2(n))\right].
\end{equation*}
Therefore,
\begin{align*}
\frac{\tau_n^\times}{\dot{\tau}_n}
	= \frac{\lambda_n\beta_n + b\beta_n'}
		{\lambda_n\alpha_n\left(1 + b\frac{\alpha_n'}{\lambda_n\alpha_n}\right)}
	&= \left(\frac{\beta_n}{\alpha_n} + b\frac{\beta_n'}{\lambda_n\alpha_n}\right)
		\left[1 + \bO(n^{-1/3}\omega_r(n))\right]
	\\
	&= \frac{\beta_n}{\alpha_n} + b\frac{\beta_n'}{\lambda_n\alpha_n}
		+ \bO(n^{-1/3}\omega_r^2(n)),
\end{align*}
which leads to the result.
\end{proof}

The function $g_A(z) = \exp(-\tfrac23\re z^{3/2})$ is defined in terms of the principal branch of the square
root, that is,
\begin{equation*}
g_A(r e^{i\theta}) = \exp(-\tfrac23 r^{3/2} \cos(\tfrac32\theta)),\quad \theta\in(-\pi,\pi].
\end{equation*}
Therefore, in order to evaluate $g_A(-z)$, we must choose the argument of $-z$ according to the rule
\begin{equation*}
-z = \abs{z} e^{i\eta}\quad\text{with}\quad
\eta = 	\begin{dcases}
		\arg(z) +\pi,& \text{if }\arg(z)\in(-\pi,0],
		\\
		\arg(z) -\pi,& \text{if }\arg(z)\in(0,\pi].
		\end{dcases}
\end{equation*}
Thus,
\begin{equation}
\label{eq:bound-gA}
g_A(-z)
	= \begin{dcases}
		\exp(\tfrac23\im z^{3/2}),& \text{if }\arg(z)\in(0,\pi],
		\\
		\exp(-\tfrac23\im z^{3/2}),& \text{if }\arg(z)\in(-\pi,0].
	  \end{dcases}
\end{equation}
Alternatively,
\begin{equation}
\label{eq:bound-gA-alt}
g_A(-z)
	= \begin{dcases}
		\exp(\tfrac23\abss{\im z^{3/2}}),& \text{if }\arg(z)\in(-\tfrac23\pi,\tfrac23\pi],
		\\
		\exp(-\tfrac23\abss{\im z^{3/2}}),& \text{if }\arg(z)\in(-\pi,-\tfrac23\pi]\cup(\tfrac23\pi,\pi].
	  \end{dcases}
\end{equation}

\begin{lemma}
\label{lem:another-bound-on-gA}
There exists $C_A>0$ such that
\begin{equation*}
	g_A(-(z + \abs{z}^{-1/2}e^{it}))
	\le C_A g_A(-z),
\end{equation*}
for all $z\in\C$ such that $\abs{z}\ge 2$ and $t\in[0,2\pi]$.
\end{lemma}
\begin{proof}
Let us define
\begin{gather*}
D_+ \defeq \left\{re^{i\theta} : r\ge 2\land\theta\in(0,\pi]\right\},\quad
D_- \defeq \left\{re^{i\theta} : r\ge 2\land\theta\in(-\pi,0]\right\},
\\
D_R \defeq \left\{re^{i\theta} : r\ge 2\land\theta\in(-\tfrac23\pi,\tfrac23\pi]\right\}
	\;\text{ and }\;
D_L \defeq \left\{re^{i\theta} : r\ge 2\land\theta\in(-\pi,-\tfrac23\pi]\cup(\tfrac23\pi,\pi]\right\}.
\end{gather*}
Consider $z = \abs{z}e^{i\theta}$ with $\abs{z}\ge 2$. Then, it is not difficult to see that the curve
\begin{equation*}
C_z\defeq\left\{z + \abs{z}^{-1/2}e^{it}: t\in[0,2\pi]\right\}
\end{equation*}
lies entirely within $D_+$, $D_-$, $D_R$ or $D_L$.
Let us assume $C_z\subset D_+$. Then, in view of \eqref{eq:bound-gA},
\begin{equation*}
g_A(-(z + \abs{z}^{-1/2}e^{it}))
	= e^{\frac23\im(z + \abs{z}^{-1/2}e^{it})^{3/2}}.
\end{equation*}
Clearly,
\begin{equation*}
z + \abs{z}^{-1/2}e^{it} = z[1 + \abs{z}^{-3/2}e^{i(t-\theta)}]
\quad\text{and}\quad
1 + \abs{z}^{-3/2}e^{i(t-\theta)} = \kappa e^{i u},
\end{equation*}
where
\begin{equation*}
\kappa
	= \sqrt{1 + 2 \abs{z}^{-3/2}\cos(t-\theta) + \abs{z}^{-3}}
	= 1 + \bO(\abs{z}^{-3/2})
\end{equation*}
as $\abs{z}\to\infty$ (we will tacitly assume this limit from now on). Also,
\begin{equation*}
\tan u
	= \abs{z}^{-3/2} \frac{\sin(t-\theta)}{1 + \abs{z}^{-3/2} \cos(t-\theta)}
	= \abs{z}^{-3/2} \sin(t-\theta) [1 + \bO(\abs{z}^{-3/2})],
\end{equation*}
which in turn implies
\begin{equation*}
u = \abs{z}^{-3/2} \sin(t-\theta) [1 + \bO(\abs{z}^{-3/2})].
\end{equation*}
Therefore,
\begin{equation*}
\im(z + \abs{z}^{-1/2}e^{it})^{3/2}
	= \abs{z}^{3/2}\kappa^{3/2}\sin(\tfrac32\theta + \tfrac32u)
	= \abs{z}^{3/2}[1 + \bO(\abs{z}^{-3/2})]\sin(\tfrac32\theta + \tfrac32u).
\end{equation*}
Moreover,
\begin{align*}
\sin(\tfrac32\theta + \tfrac32u)
	&= \sin(\tfrac32\theta)\cos(\tfrac32u) + \cos(\tfrac32\theta)\sin(\tfrac32u)
	\\
	&=  \sin(\tfrac32\theta)[1 + \bO(\abs{z}^3)]
		+ \tfrac32 \abs{z}^{-3/2}\cos(\tfrac32\theta)\sin(t-\theta)[1 + \bO(\abs{z}^{-3/2})].
\end{align*}
Thus,
\begin{equation*}
\im(z + \abs{z}^{-1/2}e^{it})^{3/2}
	= \left[\abs{z}^{3/2}\sin(\tfrac32\theta) + \tfrac32\cos(\tfrac32\theta)\sin(t-\theta)\right]
		[1 + \bO(\abs{z}^{-3/2})],
\end{equation*}
implying the existence of a positive constant $G_+$ such that
\begin{equation*}
g_A(-(z + \abs{z}^{-1/2}e^{it}))
	\le G_+ e^{\frac23\im{z^{3/2}}},\quad t\in[0,2\pi],
\end{equation*}
for all $z$ such that $C_z\subset D_+$ (and $\abs{z}\ge 2$). Analogously, there exists $G_- > 0$ such
that
\begin{equation*}
g_A(-(z + \abs{z}^{-1/2}e^{it}))
	\le G_- e^{-\frac23\im{z^{3/2}}},\quad t\in[0,2\pi],
\end{equation*}
for all $z$ such that $C_z\subset D_-$. Also, now resorting to \eqref{eq:bound-gA-alt}, either
\begin{equation*}
g_A(-(z + \abs{z}^{-1/2}e^{it}))
	\le G_R e^{\frac23\abs{\im{z^{3/2}}}}
\quad\text{or}\quad
g_A(-(z + \abs{z}^{-1/2}e^{it}))
	\le G_L e^{-\frac23\abs{\im{z^{3/2}}}},
\end{equation*}
whenever $C_z\subset D_R$ or $C_z\subset D_L$, respectively. The assertion now follows.
\end{proof}

We refer to Remark~\ref{rem:omegas} for the definition of $\Omega_r$ and its relation to $\om_r$
and certain positive constant $C_r$.

\begin{lemma}
\label{lem:z-derivatives}
Given a bounded subset $\cU\subset\sA_r$, let us define  $C_\cU\defeq C_r\sup\{\norm{q}_{\cA_r} : q\in\cU\}$.
Suppose $\abs{z}\ge 2$. Then,
\begin{equation}
\label{eq:derivatives-Xi}
	\abs{\partial_z^n\varXi(q,z,0)}
	\le  n!C_\cU e^{C_\cU\Omega_r(z - \abs{z}^{-1/2})}
		\frac{\abs{z}^{n/2}}{\sigma(z - \abs{z}^{-1/2})}\Omega_r(z - \abs{z}^{-1/2})g_A(-z)
\end{equation}
and
\begin{equation}
\label{eq:derivatives-Xi-prime}
	\abs{\partial_z^n\varXi'(q,z,0)}
	\le n!C_\cU e^{C_\cU\Omega_r(z - \abs{z}^{-1/2})}
		\abs{z}^{n/2}\sigma(z + \abs{z}^{-1/2})\Omega_r(z - \abs{z}^{-1/2})g_A(-z).
\end{equation}
for all $q\in\cU$.
\end{lemma}
\begin{proof}
We recall that
\begin{equation*}
	\varXi(q,z,0) = \sum_{k=1}^\infty \psi_k(q,z,0),
\end{equation*}
where the terms of this series are defined in the proof of Lemma~3.1 in \cite{toluri-2022}
and the convergence is uniform on compact subsets of $\C$. Therefore,
\begin{equation*}
	\partial_z^n\varXi(q,z,0) = \sum_{k=1}^\infty \partial_z^n\psi_k(q,z,0).
\end{equation*}
For suitable $r(z)>0$,
\begin{equation*}
	\partial_z^n\psi_k(q,z,0)
		= \frac{n!}{2\pi r(z)^j}\int_0^{2\pi}\psi_k(q,z+r(z)e^{it},0)e^{-int}dt,
\end{equation*}
hence
\begin{equation*}
\abs{\partial_z^n\psi_k(q,z,0)}
	\le \frac{n!}{r(z)^n} \max_{t\in[0,2\pi]}\abs{\psi_k(q,z+r(z)e^{it},0)}.
\end{equation*}
Moreover,
\begin{equation*}
\abs{\psi_k(q,w,0)}
	\le \frac{4^k}{k!}C_0^{2k+1} \frac{g_A(-w)}{\sigma(w)} \omega(q,w)^k.
\end{equation*}
From now on we assume $\abs{z}\ge 2$ and set $r(z) = \abs{z}^{-1/2}$. Noting that
\begin{equation*}
\sigma(z - \abs{z}^{-1/2})
	\le \sigma(z + \abs{z}^{-1/2} e^{it})
	\le \sigma(z + \abs{z}^{-1/2}),
\quad
\Omega_r(z + \abs{z}^{-1/2} e^{it})
	\le \Omega_r(z - \abs{z}^{-1/2}),
\end{equation*}
and recalling Lemma~\ref{lem:another-bound-on-gA}, we obtain
\begin{equation*}
\abs{\partial_z^n\psi_k(q,\lambda,0)}
	\le 4^k \frac{n!}{k!} C_A C_0^{2k+1} C_\cU^k
		\frac{\abs{z}^{n/2}}{\sigma(z - \abs{z}^{-1/2})}g_A(-z)
		\Omega_r(z - \abs{z}^{-1/2})^k.
\end{equation*}
Then, \eqref{eq:derivatives-Xi} follows immediately (after redefining the constant $C_\cU$). The proof of
\eqref{eq:derivatives-Xi-prime} is nearly identical, hence omitted.
\end{proof}

\begin{remark}
\label{rem:some-asymptotics}
As a consequence, we have
	\begin{gather*}
		\varXi_n = \bO(n^{-1/6}\om_r(n)),
		\quad
		\dot{\varXi}_n = \bO(n^{1/6}\om_r(n)),
		\quad
		\ddot{\varXi}_n = \bO(n^{1/2}\om_r(n)),
		\\
		\varXi_n' = \bO(n^{1/6}\om_r(n)),
		\quad
		\dot{\varXi}_n' = \bO(n^{1/2}\om_r(n)),
		\quad
		\ddot{\varXi}_n' = \bO(n^{5/6}\om_r(n)),
	\end{gather*}
uniformly for $(q,b)$ on bounded subsets of $\bsA_r\times\R$.
\end{remark}

\begin{remark}
\label{rem:refinement}
Let us recall the following identities, introduced in \cite{toluri-2022},
\begin{equation*}
\varXi(q,z,x) = \psi_1(q,z,x) + \varXi^{(2)}(q,z,x),
\end{equation*}
where
\begin{equation*}
\psi_1(q,z,x)
	\defeq - \int_x^\infty J_0(z,x,y)\psi_0(z,y)q(y)dy
\end{equation*}
and
\begin{equation*}
\abs{\varXi^{(2)}(q,z,x)}
	\le C \om^2(q,z) e^{C \om(q,z)}\frac{g_A(x-z)}{\sigma(x-z)}.
\end{equation*}
Also,
\begin{equation}
\label{eq:recall-1}
	\varXi'(q,z,x) = \psi_1'(q,z,x) + \varXi^{(2)\prime}(q,z,x),
\end{equation}
where
\begin{equation*}
	\psi_1'(q,\lambda,x)
	= - \int_x^\infty \partial_xJ_0(z,x,y)\psi_0(z,y)q(y)dy
\end{equation*}
and
\begin{equation*}
\abs{\varXi^{(2)\prime}(q,z,x)}
	\le C \om^2(q,z) e^{C \om(q,z)} \sigma(x-z) g_A(x-z).
\end{equation*}
The estimates above hold true if $q\in\sA_r^\C$. Now suppose $q\in\bsA_r^\C$. Then,
\begin{equation*}
\dot{\varXi}(q,z,x) = - \psi_1'(q,z,x) + \dot{\psi}_1^\text{res}(q,z,x) + \dot{\varXi}^{(2)}(q,z,x),
\end{equation*}
where
\begin{equation*}
\dot{\psi}_1^\text{res}(q,z,x)
	\defeq - \int_x^\infty J_0(z,x,y)\psi_0(z,y)q'(y)dy,
\end{equation*}
and
\begin{equation*}
\abs{\dot{\varXi}^{(2)}(q,z,x)}
	\le C \bom^2(q,z) e^{C \bom(q,z)} \sigma(x-z) g_A(x-z).
\end{equation*}
Moreover,
\begin{equation}
\label{eq:recall-2}
\dot{\varXi}'(q,z,x)
	= - \psi_0(z,x)q(x) + z\psi_1(q,z,x)
	+ \dot{\psi'}_1^\text{res}(q,z,x)
	+ \dot{\varXi'}^{(2)}(q,z,x),
\end{equation}
where
\begin{equation*}
\dot{\psi'}_1^\text{res}(q,z,x)
	= -\int_x^\infty\partial_xJ_0(z,x,y)\psi_0(z,y)q'(y)dy.
\end{equation*}
Furthermore,
\begin{equation*}
\ddot{\varXi}(q,z,x)
	= \psi_0(z,x)q(x) - z \psi_1(q,z,x) - \dot{\psi'}_1^\text{res}(q,z,x)
	+ \partial_z\dot{\psi}_1^\text{res}(q,z,x) + \ddot{\varXi}^{(2)}(q,z,x).
\end{equation*}
Finally,
\begin{align*}
\ddot{\varXi'}(q,z,x)
	=&\, \psi'_0(z,x)q(x) + \psi_1(q,z,x)
	\\
	 & - z\psi'_1(q,z,x) + z \dot{\psi}_1^\text{res}(q,z,x)
	   + \partial_z\dot{\psi'}_1^\text{res}(q,z,x) + \ddot{\varXi'}^{(2)}(q,z,x).\qedhere
\end{align*}
\end{remark}

\begin{remark}
\label{rem:As-and-Bs}
Assume $q\in\bsA_r$ and define
\begin{gather*}
P_n \defeq \int_{0}^{\infty}\psi_0(\lambda_n,x)^2q(x) dx,
\quad
Q_n \defeq \int_{0}^{\infty}\psi_0(\lambda_n,x)\theta_0(\lambda_n,x)q(x) dx,
\\
P'_n \defeq \int_{0}^{\infty}\psi_0(\lambda_n,x)^2q'(x) dx,
\quad
Q'_n \defeq \int_{0}^{\infty}\psi_0(\lambda_n,x)\theta_0(\lambda_n,x)q'(x) dx.
\end{gather*}
Clearly all these expressions are $\bO(\om_r(n))$. Moreover,
\begin{gather*}
\psi_{1,n} = - \beta_n P_n + \alpha_n Q_n,\quad
\psi'_{1,n} = - \beta'_n P_n + \alpha'_n Q_n,
\\
\dot{\psi}_{1,n}^\text{res} = - \beta_n P'_n + \alpha_n Q'_n,\quad
\dot{\psi'}_{1,n}^\text{res} = - \beta'_n P'_n + \alpha'_n Q'_n.\qedhere
\end{gather*}
\end{remark}

\begin{remark}
\label{rem:more-of-the-same}
	The proof of Lemma~\ref{lem:z-derivatives} also yields the inequalities
	\begin{equation*}
	\abs{\partial_z^n\varXi^{(2)}(q,z,0)}
	\le  n!C_\cU e^{C_\cU\Omega_r(z - \abs{z}^{-1/2})}
		\frac{\abs{z}^{n/2}}{\sigma(z - \abs{z}^{-1/2})}
			\Omega_r^2(z - \abs{z}^{-1/2})g_A(-z)
	\end{equation*}
	and
	\begin{equation*}
	\abs{\partial_z^n\varXi'^{(2)}(q,z,0)}
	\le n!C_\cU e^{C_\cU\Omega_r(z - \abs{z}^{-1/2})}
		\abs{z}^{n/2}\sigma(z + \abs{z}^{-1/2})
		\Omega_r^2(z - \abs{z}^{-1/2})g_A(-z).
	\end{equation*}
	Recalling that $\varXi_n^{(2)}=\varXi_n^{(2)}(q,\lambda_{n}(q,b),0)$ et cetera,
	these inequalities implies
	\begin{gather*}
		\varXi_n^{(2)} = \bO(n^{-1/6}\om^2_r(n)),
		\quad
		\dot{\varXi}_n^{(2)} = \bO(n^{1/6}\om^2_r(n)),
		\quad
		\ddot{\varXi}_n^{(2)} = \bO(n^{1/2}\om^2_r(n)),
		\\
		\varXi_n'^{(2)} = \bO(n^{1/6}\om^2_r(n)),
		\quad
		\dot{\varXi}_n'^{(2)} = \bO(n^{1/2}\om^2_r(n)),
		\quad
		\ddot{\varXi}_n'^{(2)} = \bO(n^{5/6}\om^2_r(n)),
	\end{gather*}
uniformly for $(q,b)$ on bounded subsets of $\bsA_r\times\R$.
\end{remark}

\begin{remark}
\label{rem:derivatives-psi-1}
By comparing $\dot{\psi}_1^\text{res}$ and $\dot{\psi'}_1^\text{res}$ with
$\psi_1$ and $\psi'_1$, respectively, we obtain the estimates
\begin{equation*}
\abs{\partial_z^n\dot{\psi}_1^\text{res}(q,\lambda,0)}
	\le n! C_\cV
		\frac{\abs{z}^{n/2}}{\sigma(z - \abs{z}^{-1/2})}\Omega_r(z - \abs{z}^{-1/2})g_A(-z).
\end{equation*}
and
\begin{equation*}
	\abs{\partial_z^n\dot{\psi'}_1^\text{res}(q,\lambda,0)}
	\le n! C_\cV
	\abs{z}^{n/2}\sigma(z - \abs{z}^{-1/2})
	\Omega_r(z - \abs{z}^{-1/2})g_A(-z),
\end{equation*}
but now assuming $q\in\cV$, a bounded subset of $\bsA_r$. Consequently,
\begin{equation*}
\partial_z\dot{\psi}_{1,n}^\text{res}
	= \bO(n^{1/6}\om_r(n))
\quad\text{and}\quad
\partial_z\dot{\psi'}_{1,n}^\text{res}
	= \bO(n^{1/2}\om_r(n)),
\end{equation*}
uniformly on bounded subsets of $\bsA_r\times\R$.
\end{remark}

\begin{remark}
\label{rem:the-end}
For the sake of brevity, let us just suppose $q\in\bsA_r^\C$. We have the following identities:
\begin{align*}
\phi(q,b,z,x)
	=&\, \phi_0(b,z,x) + \phi_1(q,b,z,x) + \varPhi^{(2)}(q,b,z,x),
	\\
\phi'(q,b,z,x)
	=&\, \phi'_0(b,z,x) + \phi'_1(q,b,z,x) + {\varPhi}^{(2)\prime}(q,b,z,x)
\intertext{and}
\dot{\phi}(q,b,z,x)
	=&\, \phi^\times_0(b,z,x) - \phi_0'(b,z,x) + s_0(z,x) q(0) + \phi^\times_1(q,b,z,x) - \phi_1'(q,b,z,x)
	\\
	 & + \dot{\phi}^\text{res}_1(q,b,z,x) + \dot{\varPhi}^{(2)}(q,b,z,x),
\end{align*}
where
\begin{align*}
\phi_0^\times(b,z,x)	&= -zs_0(z,x) + b c_0(z,x),
	\\
\phi_1(q,b,z,x) 		&= \int_0^x J_0(z,x,y)\phi_0(b,y)q(y) dy,
	\\
\phi^\times_1(q,b,z,x)	&= \int_0^x J_0(z,x,y)\phi^\times_0(b,y)q(y) dy
\intertext{and}
\dot{\phi}^\text{res}_1(q,b,z,x) &= \int_0^x J_0(z,x,y)\phi_0(b,y)q'(y) dy.
\end{align*}
Besides,
\begin{equation*}
\varPhi^{(2)}(q,b,z,x)
	= \varUpsilon_c^{(2)}(q,z,x) + b \varUpsilon_s^{(2)}(q,z,x),
\end{equation*}
where $\varUpsilon_c^{(2)}$ and $\varUpsilon_s^{(2)}$ are defined in Remark~3.11 of \cite{toluri-2022}.
Analogous definitions hold for $\varPhi^{(2)\prime}$ and $\dot{\varPhi}^{(2)}$.
We note that
\begin{equation}
\label{eq:phi-1}
\abs{\phi_1(q,b,z,x)}
	\le C(1+\abs{b}) \bom(q,z) \frac{\sigma(z)}{\sigma(x-z)}\ch(z,x)
\end{equation}
and the same bound holds for $\dot{\phi}^\text{res}_1$. Also, Remark~3.11 of \cite{toluri-2022}
implies the following bounds:
\begin{gather}
\abs{\varPhi^{(2)}(q,b,z,x)}
	\le C(1+\abs{b})\bom^2(q,z)e^{C\bom(q,z)}
		\left(\sigma(z) + \frac{1}{\sigma(z)}\right)\frac{\ch(z,x)}{\sigma(x-z)},\nonumber
\\
\abs{\varPhi^{(2)\prime}(q,b,z,x)}
	\le C(1+\abs{b})\bom^2(q,z)e^{C\bom(q,z)}
		\left(\frac{\sigma(z)}{\sigma(x-z)} + \frac{\sigma(x-z)}{\sigma(z)}\right)\ch(z,x)\nonumber
\intertext{and}
\label{eq:Phi-dot-2}
\abs{\dot{\varPhi}^{(2)}(q,b,z,x)}
	\le C(1+\abs{b})\bom^2(q,z) e^{C\bom(q,z)}
		\left(\sigma(z)\sigma(x-z) + \frac{\abs{q(0)} + \abs{z}}{\sigma(z)\sigma(x-z)}\right)\ch(z,x).
\end{gather}
No further comment is required.
\end{remark}


\end{document}